\documentclass[11pt]{article}

\usepackage{titlesec}

\titleformat{\section}
  {\normalfont\large\bfseries}{\thesection}{1em}{}

\titleformat{\subsection}
  {\normalfont\normalsize\itshape}{\thesubsection}{1em}{}

\titleformat{\subsubsection}
  {\normalfont\normalsize\itshape}{\thesubsubsection}{1em}{}
  
\usepackage{graphicx}
\usepackage{float}
\usepackage{booktabs,array}
\usepackage{amsfonts,amsmath,amssymb}
\usepackage{natbib}
\usepackage[colorlinks=true,linkcolor=blue,citecolor=blue,urlcolor=blue]{hyperref}
\usepackage{setspace}
\usepackage[font=small,labelfont=bf]{caption}
\usepackage{etoolbox}
\usepackage{adjustbox}
\newcolumntype{L}[1]{>{\raggedright\arraybackslash}p{#1}}

\DeclareMathOperator{\E}{\mathbb{E}}

\usepackage{tikz}
\usetikzlibrary{arrows.meta,calc,positioning}
\makeatletter

\@ifundefined{@combinedblfloats}{}{%
  \patchcmd\@combinedblfloats{\box\@outputbox}{\unvbox\@outputbox}{}{}%
}%
\makeatother

\newif\iflatexml\latexmlfalse

\AtBeginDocument{\DeclareGraphicsExtensions{.pdf,.PDF,.eps,.EPS,.png,.PNG,.tif,.TIF,.jpg,.JPG,.jpeg,.JPEG}}

\usepackage[utf8]{inputenc}
\usepackage[english]{babel}

\usepackage[letterpaper,left=1in,right=1in,top=1in,bottom=1in]{geometry}

\usepackage{amsthm}
\newtheorem{proposition}{Proposition}

\newtheorem{theorem}{Theorem}

\newtheorem{remark}{Remark}

\newtheoremstyle{upright}
  {6pt}   
  {6pt}   
  {\normalfont}  
  {}      
  {\bfseries} 
  {.}     
  { }     
  {}
\theoremstyle{upright}

\newtheorem{exmp}{Example}[section]
\newtheorem{exmpm}{Example}

\newcounter{fstmt}
\renewcommand{\thefstmt}{F\arabic{fstmt}}

\newenvironment{fstatement}[1][]%
{%
  \refstepcounter{fstmt}%
  \par\medskip\noindent
  \textbf{(\thefstmt) #1}\@\enspace\ignorespaces
}%
{%
  \par\medskip
}

\iflatexml

\else
\fi

\title{One Probability, Two Roles: The Separation of Coherence and Frequency in Adaptive Regimes}

\author{Yonggang Lu\thanks{Contact email: \href{mailto:yonggang.t.lu@gmail.com}{\texttt{yonggang.t.lu@gmail.com}}}\\University of Maine}

\date{September 1, 2026}

\hypersetup{
  pdftitle={One Probability, Two Roles: The Separation of Coherence and Frequency in Adaptive Regimes},
  pdfauthor={Yonggang Lu},
  pdfsubject={Probabilistic coherence and regime-conditional frequency guarantees},
  pdfkeywords={probabilistic coherence, frequency interpretation, adaptive regimes, reference structures, stabilization, calibration, exchangeability}
}

\begin{document}

\maketitle

\begin{center}
\small\textit{This preprint is a work in progress and may contain errors. Comments and suggestions are welcome.}
\end{center}
\medskip
\selectlanguage{english}

\begin{abstract}
\noindent Probability plays two distinct roles in modern data-analytic practice: (i) as an internally coherent, filtration-relative language for sequential forecasting and, together with a stated loss or utility, decision making; and (ii) as a foundation for empirical claims such as stabilization, calibration, and repeated-sampling validity. The first role is relative to an assessment law $\mathbb Q$, whereas the second is evaluated under a governing law $\mathbb P$ and a declared repetition regime. In classical settings the roles are often aligned by conditional correctness under $\mathbb P$ together with i.i.d.\ or other stable-law assumptions under a fixed design, or by exchangeability. In adaptive and AI-mediated settings they can diverge because forecasts operate within feedback loops while reliability is evaluated under a potentially different law and regime. We synthesize relevant results from prequential forecasting, calibration, martingale and game-theoretic validity, adaptive data analysis, conformal prediction, and long-context AI reliability into a formal account of \emph{role separation}. The resulting \emph{Role Separation Principle} distinguishes internal coherence from conditional-mean correctness under the evaluation law and from target-specific reference and stability conditions. Two separation results show that a proper coarsening of the limiting-frequency distribution need not identify the governing law, and that even conditionally correct forecasts under that law need not accompany stabilization. We then develop a unified stabilization framework spanning filtration-based, shift-based, policy--environment, and design-based regimes and operationalize it through a regime-conditional audit checklist. Applied illustrations clarify calibration claims, synthetic-data validity, and structural mismatches relevant to some high-confidence unsupported AI outputs.

\medskip
\noindent \textbf{Keywords} --- Probabilistic coherence, Frequency interpretation, Adaptive regimes, Reference structures, Stabilization, Calibration, Exchangeability, AI-mediated systems
\end{abstract}%

\section{Introduction}
\label{sec1-intro}

In many contemporary deployments, particularly in reinforcement learning, adaptive experimentation, and generative AI, data are produced within closed-loop systems in which prediction and action shape future observations. Recommendation and ranking algorithms alter information exposure \citep{Bottou2013,ChaneyStewartEngelhardt2018}; adaptive clinical trials can update assignment rules as outcomes accrue \citep{RobertsonEtAl2023}; and reinforcement-learning agents select actions to influence subsequent trajectories \citep{SuttonBarto2018}. In such regimes, the ``data-generating mechanism'' is not a fixed background object but part of the deployed system, creating feedback between prediction, action, and observation.

This closed-loop mode of data generation complicates a familiar but often implicit identification between internal probabilistic consistency and external reliability. In static textbook settings, strong invariance assumptions, most notably i.i.d.\ sampling from a stable population, align the meaning of probabilistic statements with their validation through repetition. In adaptive regimes, where policies evolve, environments respond, and distributions drift endogenously, this alignment can fail. As a result, it becomes unclear which inferential meanings remain stable and which require explicit assumptions about repetition and structural invariance.

One way to sharpen the issue is to recognize that probability is routinely asked to play two distinct roles in contemporary data-analytic practice. In the coherence tradition, probability is an interpretive and normative device that supports coherent belief updating and, together with a stated loss or utility, decision making via conditioning. While coherence is often motivated in finitely additive terms \citep{Savage1954,deFinetti1974}, throughout we work in standard measure-theoretic probability, under which assessments are represented by a countably additive probability measure. In the frequency tradition, probability supports interpretations grounded in repetition and invariance, such as long-run relative frequencies, coverage, and error control \citep{vonMises1928,Neyman1937,NeymanPearson1933}. Contemporary workflows rely on both: models are deployed for real-time forecasting and decision support, while predictive calibration is evaluated empirically alongside task-specific performance \citep{Guo2017}.

Adaptive regimes expose a structural fault line between these roles. Internal coherence is defined relative to an information filtration and an assessment law, whereas empirical reliability is evaluated under a governing path law and presupposes a specified repetition regime. A system may therefore be internally coherent while factual accuracy, downstream inferential validity, or deployment reward is evaluated under a different law. This mismatch is relevant to reliability failures in generative AI (often termed ``hallucinations'') \citep{Ji2023}. In synthetic-data settings, the corresponding issue is transfer from the generator law to the real-data law, which requires assumptions not supplied merely by coherent sampling from the generator \citep{Rubin1993,Reiter2005,Drechsler2011}. Recent empirical work supplies complementary motivation. Models may fail to use information robustly across positions in a long context \citep{Liu2024}, degrade over extended conversations \citep{Laban2025}, and exhibit turn-wise KL divergence between a test model and a goal-consistent reference model \citep{DongreEtAl2025}. These findings motivate, but do not replace, the regime-level question addressed here:
\begin{quote}
\textit{What assumptions are needed before locally meaningful forecasts support stable frequency interpretations over long interactions}?
\end{quote}

Against this backdrop, rather than proposing another calibration algorithm, sequential test, conformal construction, or empirical benchmark for model drift, we develop a regime-conditional mathematical framework that explains why locally coherent probabilistic behavior may nevertheless fail to support long-run frequency reliability in adaptive regimes and why stabilization must be tied to an explicit reference structure. Conceptually, we separate coherence from frequency meaning and emphasize that claims of probabilistic reliability require structural assumptions specifying what is held fixed under repetition. Structurally, we formalize this separation as a \emph{Role Separation Principle} and develop a unified stabilization framework that connects the canonical exchangeable case to shift-invariant, policy--environment, and design-based regimes. Practically, we translate these distinctions into a regime-conditional audit framework that supports transparent and structurally grounded interpretation of probabilistic guarantees in AI-mediated systems. Throughout, our aim is semantic clarification and principled inference under declared regimes rather than the proposal of new testing procedures or calibration algorithms.

This novelty is deliberately synthetic. The two separation facts we emphasize have classical roots in exchangeability, ergodic theory, and martingale methods, and closely related literatures already supply powerful tools for sequential forecasting, calibration, testing, and adaptive validity. What is missing, and what this paper supplies, is a unified reference-structure account of when filtration-level probabilistic coherence can be interpreted as a frequency-level guarantee, and when it cannot. We summarize the central message as the following Role Separation Principle (henceforth, the Principle):
\begin{quote}
\textit{Internal coherence makes probability meaningful as a filtration-relative conditional assessment under its assessment law. Empirical guarantees additionally require conditional correctness under the governing evaluation law, when such correctness is invoked, and target-specific assumptions specifying repetition and stability.}
\end{quote}
This principle is structural rather than philosophical. It clarifies what coherence alone can guarantee in adaptive settings and what additional assumptions are needed before probabilistic statements support stabilization, calibration, or repeated-sampling error guarantees. Our approach is synthetic in that it organizes standard results from conditional expectation, martingale theory, exchangeability, and ergodic stabilization into a unified diagnostic framework.

At the technical level, any deployed forecasting procedure induces a process $(\mathcal F_t,p_t,X_{t+1})_{t\ge0}$, where $\mathcal F_t$ is the accumulated information, $p_t$ is an $\mathcal F_t$-measurable assessment, and $0\le X_{t+1}\le1$ is the evaluated outcome. Let $\mathbb Q$ denote an assessment law and $\mathbb P$ the governing evaluation law. Internal coherence under $\mathbb Q$ means $p_t=\mathbb E_{\mathbb Q}[X_{t+1}\mid\mathcal F_t]$. The empirical martingale conclusions in this paper hold under $\mathbb P$ only when the stronger deployment condition $p_t=\mathbb E_{\mathbb P}[X_{t+1}\mid\mathcal F_t]$ holds $\mathbb P$-almost surely. We call the latter \emph{conditional-mean correctness under the evaluation law}. It implies that $(X_{t+1}-p_t)$ is a bounded $\mathbb P$-martingale-difference sequence without requiring independence, stationarity, or exchangeability. We use exact conditional correctness as a benchmark; we do not assume that a particular AI training procedure enforces it.

Under evaluation-law conditional correctness, a martingale strong law implies that average forecast error converges to zero. This is the baseline unweighted martingale conclusion. It does not imply stabilization of empirical averages or beliefs, guarantee visitation of forecast bins, or provide uniform calibration over rich checking classes. For a fixed predictable Borel bin, however, the same martingale structure does yield a calibration gap converging to zero on the event of infinite visitation. This distinction matters in adaptive regimes, where ``calibration'' is easily interpreted more broadly than the available theory supports \citep{Dawid1982,FosterVohra1998}. Dawid relates coherent Bayesian forecasting to calibration within a specified sequential protocol, while Foster and Vohra obtain online calibration through an explicit forecasting game. Both support treating calibration as a property of a declared forecasting--outcome regime.

Frequency meanings, however, require more than coherence. They require a repetition mechanism: an invariance or controlled evolution that renders long-run averages comparable. We refer to the corresponding regime-level assumptions as the \emph{reference structure}. Canonical examples include permutation invariance (exchangeability), shift invariance (stationarity and ergodicity), Markovian invariance under a fixed policy--environment pair, and design-based reference structures in adaptive experimentation.

Exchangeability provides the clearest alignment case. Under the governing exchangeable law, conditionally correct one-step forecasts and empirical frequencies converge to the same latent quantity. Beyond such settings, we establish two complementary separation results: (i) a proper coarsening of the limiting-frequency distribution need not identify the governing law, although the full limiting-frequency distribution does identify a binary exchangeable law and a single realized limit supplies only one observation from that distribution; and (ii) evaluation-law conditional correctness does not ensure almost sure stabilization of empirical averages. To characterize when stabilization does occur, we introduce a \emph{reference--drift decomposition} that writes empirical averages as a predictable drift term plus a martingale innovation term. Stabilization then reduces to verifying drift convergence under the stated regime assumptions.

The perspective developed here is closely related to prequential and sequential forecasting traditions, which evaluate probabilistic models through realized predictive performance and emphasize calibration, martingale behavior, and testing procedures for pathwise validity \citep{Dawid1982,Dawid1984,Dawid2006,FosterVohra1998}. It is also related to game-theoretic probability, in which validity is expressed through betting protocols and supermartingale capital processes without classical independence assumptions \citep{ShaferVovk2001}. These literatures provide the sequential and martingale tools on which our formulation builds.

Our focus is complementary but distinct. Where prequential and game-theoretic approaches emphasize scoring, betting, and pathwise testing, we emphasize the semantic separation between \textit{filtration-level} coherence and \textit{regime-level} frequency interpretation and make explicit both the required assessment--evaluation law relation and the invariance or controlled-evolution assumptions. This distinction is especially salient when the repetition mechanism underlying frequency claims is implicit, evolving, strategically altered, or mediated by an AI system whose effective context may differ from the context users attribute to it.

For adaptive and AI-mediated regimes, the Principle has a concrete methodological message. Probabilistic claims should report (i) the assessment and evaluation laws when they differ, (ii) the filtration and target event, and (iii) the reference structure under which the intended frequency meaning is asserted. Misalignment among these elements offers diagnostic hypotheses for some high-confidence unsupported outputs and synthetic-data transfer failures, but the formal results do not establish a universal empirical cause or remedy. We therefore translate the theory into auditable criteria for reliability claims under feedback and drift.

The remainder of the paper proceeds as follows. \autoref{sec2-map} introduces a conceptual map distinguishing coherence, multiple frequency meanings, and reference structures in adaptive regimes. \autoref{sec3-coherence-forecast} formalizes the coherence layer and its interpretation as filtration-relative conditional forecasting and decision making. \autoref{sec4-exchangeability} presents exchangeability as a canonical symmetry-based reference structure under which governing-law conditionally correct forecasts and frequency align through a latent representation. \autoref{sec5-separations} establishes two separation results that clarify the limits of frequency identification and stabilization under conditional correctness alone. \autoref{sec6-beyond-exch} develops a unified stabilization framework beyond exchangeability, organizing shift-based, policy--environment, and design-based regimes via the reference--drift decomposition. \autoref{sec7-guidance} provides interpretive guidance for adaptive and AI-mediated settings, including high-confidence errors often described as hallucinations, calibration, and synthetic data, and develops regime-conditional audit criteria for articulating probabilistic guarantees. \autoref{sec8-conclusion} concludes with a synthesis of the main findings and implications for probabilistic reasoning in modern adaptive environments.

\section{A Conceptual Map of Coherence, Frequency, and Reference Structures}
\label{sec2-map}

We begin by summarizing the Principle introduced in \autoref{sec1-intro} at a conceptual level, as illustrated in \autoref{fig:sep-map}. 
\begin{figure}[htb]
\centering
\begin{tikzpicture}[
    scale=0.76, transform shape,
    box/.style={rectangle, draw, rounded corners, align=center, text width=4.25cm, minimum height=0.9cm, font=\small},
    dashedbox/.style={rectangle, draw, dashed, rounded corners, align=center, text width=4.25cm, minimum height=0.9cm, font=\small},
    arrow/.style={-Latex, thick},
    dashedarrow/.style={-Latex, thick, dashed},
    node distance=0.75cm and 1.25cm
]
\node[box] (assessment) {Assessment law $\mathbb Q$ and filtration $\mathcal F_{t-1}$};
\node[box, below=of assessment] (coherence) {Internal coherence\\ $p_{t-1}=\mathbb E_{\mathbb Q}[X_t\mid\mathcal F_{t-1}]$};
\node[dashedbox, below=of coherence] (correctness) {Evaluation-law correctness\\ $p_{t-1}=\mathbb E_{\mathbb P}[X_t\mid\mathcal F_{t-1}]$};
\node[box, below=of correctness] (validity) {$\mathbb P$-martingale validity\\ $n^{-1}\sum_{t=1}^n(X_t-p_{t-1})\to0$};

\draw[arrow] (assessment) -- (coherence);
\draw[dashedarrow] (coherence) -- node[midway, right, font=\scriptsize, align=left] {law match or\\external verification} (correctness);
\draw[arrow] (correctness) -- (validity);

\node[dashedbox, right=of correctness] (reference) {Target-specific reference and stability conditions};
\node[box, below=of reference] (frequency) {Specified frequency guarantee\\ \footnotesize stabilization, calibration, or error control};

\draw[arrow] (reference) -- (frequency);
\draw[dashedarrow] (validity.east) -- (frequency.west);
\end{tikzpicture}

\vspace{-5pt}
\caption{\footnotesize Internal coherence is relative to an assessment law $\mathbb Q$. Martingale validity under the governing evaluation law $\mathbb P$ requires conditional-mean correctness under $\mathbb P$. A specified frequency guarantee additionally requires target-specific reference and stability conditions.}
\label{fig:sep-map}
\end{figure}
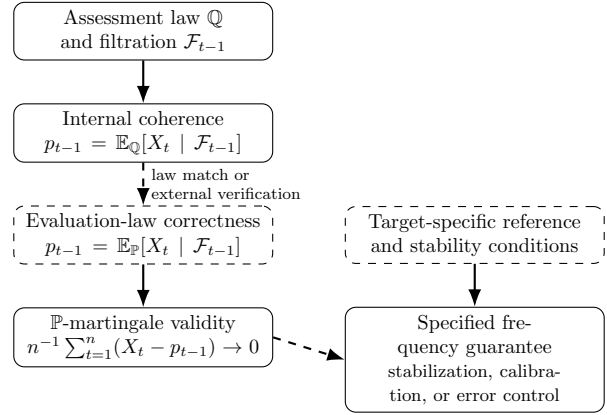

This section makes these underlying concepts precise by organizing them into a three-layer conceptual map. The first layer concerns filtration-level probabilistic meaning, interpreting coherence as conditional forecasting. The second distinguishes several frequency notions used to assess long-run reliability. The third introduces reference structures, defined by explicit invariance or controlled-evolution assumptions, that connect governing-law probabilistic statements to frequency interpretation once any required law relation has been stated. This framework serves a diagnostic role by clarifying which claims follow from coherence alone and which require evaluation-law correctness and an explicit repetition mechanism.

\subsection{Minimal probabilistic objects in adaptive regimes}
\label{sec2.1-objects}

We model an adaptive deployment as a sequentially observed process together with an information structure.
Let $(Z_t)_{t\ge1}$ denote the observed stream, where $Z_t$ may include contexts, actions, and outcomes (e.g.,
$Z_t=(C_t,A_t,X_t)$). Let
\[
\mathcal{F}_t := \sigma(Z_1,\ldots,Z_t)
\]
be the filtration representing the information available at time $t$. In many applications, actions are
chosen adaptively in the sense that $A_{t+1}$ is $\mathcal{F}_t$-measurable (or randomized conditional on
$\mathcal{F}_t$), and outcomes are then realized given the chosen action and the evolving environment.

A central object for frequency-based evaluation is a long-run empirical average of a bounded summary
statistic. For a bounded outcome sequence $(X_t)_{t\ge1}$, define
\begin{equation}
\label{eq:emp-mean}
\bar X_n := \frac{1}{n}\sum_{t=1}^n X_t,
\end{equation}
and consider the event that a long-run average exists,
\[
D := \left\{\lim_{n\to\infty}\bar X_n\ \text{exists}\right\}.
\]
In classical i.i.d.\ settings, it is routine to assert $\mathbb{P}(D)=1$ and interpret the limit as an ``underlying
probability.'' In adaptive regimes, the meaning of such an assertion depends on what is treated as fixed under
repetition (e.g., policy, environment, design, stopping rule) and what is permitted to evolve (e.g., drift induced by
feedback). These regime-level choices are precisely what we formalize as reference structures in
\autoref{sec2.4-reference-structures}.

\subsection{Coherence as filtration-level probabilistic meaning}
\label{sec2.2-coherence}

Coherence is an internal consistency requirement for probabilistic reasoning. It rules out sure-loss (Dutch-book) vulnerabilities and motivates representing a collection of assessments on a common event algebra by a probability measure \citep{Savage1954,deFinetti1975}. While coherence is often discussed in finitely additive terms, throughout we use countably additive probability. Let $\mathbb Q$ denote the assessment law used to form conditional prices and let $\mathbb P$ denote the governing law under which empirical performance is evaluated. These laws may coincide, but coherence alone does not assert that they do. A forecast issued at time $t$ is $\mathcal F_t$-measurable, where $\mathcal F_t$ records the declared information available at issuance.

A coherent assessment under $\mathbb Q$ obeys
\begin{equation}
\label{eq:coh-identity}
p_t \;=\; \mathbb{E}_{\mathbb Q}[X_{t+1}\mid \mathcal{F}_t].
\end{equation}
Thus $(X_{t+1}-p_t)$ is a martingale-difference sequence under $\mathbb Q$. To obtain the analogous empirical statement under $\mathbb P$, one must separately have
\begin{equation}
\label{eq:eval-correctness}
p_t=\mathbb E_{\mathbb P}[X_{t+1}\mid\mathcal F_t]\qquad \mathbb P\text{-a.s.}
\end{equation}
We call \eqref{eq:eval-correctness} conditional-mean correctness under the evaluation law. Neither identity requires independence, stationarity, or exchangeability, and neither is assumed to hold automatically for a deployed AI system.

Internal coherence provides an interpretation of probability as conditional assessment and, with a loss or utility, decision support. It does not establish evaluation-law correctness, specify repetition, or ensure that empirical averages stabilize.

\subsection{What ``frequency'' means in practice}
\label{sec2.3-frequency}

The term ``frequency'' is used for several distinct targets. In textbook i.i.d.\ settings these notions are often studied under a common repetition scheme, but they remain logically distinct. In adaptive regimes the distinctions become especially visible unless the repetition mechanism is stated explicitly.

\begin{fstatement}[Pathwise stabilization of empirical averages.]
\label{f:F1}
A long-run frequency exists along a realized trajectory:
\[
\bar X_n \to L \quad \text{almost surely},
\]
for some random variable $L$ (possibly constant). This is a property of a stochastic process under a specified law.
Crucially, even evaluation-law conditional correctness does not imply \hyperref[f:F1]{(\ref*{f:F1})}; later we exhibit a non-degenerate governing law with conditionally correct forecasts for which $\mathbb P(D)=0$.
\end{fstatement}

\begin{fstatement}[Repeated-sampling guarantees for procedures.]
\label{f:F2}
Coverage and type-I-error guarantees concern the behavior of a procedure across replications of a declared mechanism; time-uniform guarantees concern a process under a fixed null and reference law \citep{Ville1939,HowardRamdasMcAuliffeSekhon2020}. In adaptive settings one possible replication scheme reruns the full policy or design, including its adaptivity and stopping rule, while other schemes may condition on parts of the realized history. The chosen scheme must be stated. Regret bounds form a separate class and may be deterministic, pathwise, or expectation-based rather than repeated-sampling claims.
\end{fstatement}

\begin{fstatement}[Calibration of probabilistic forecasts.]
\label{f:F3}
Calibration compares predicted probabilities to empirical frequencies within an appropriate reference class; events forecast near $p$ should occur at empirical rate near $p$ in the long run \citep{Dawid1982,GneitingRaftery2007}.
In adaptive regimes, calibration is not a purely internal property of a forecasting rule. It is a frequency property
of the joint system generating forecasts and outcomes, and thus depends on what is held fixed (information used by the
forecaster, policy, environment stability, and evaluation protocol).
Such statements are inherently relative to a declared checking class and, in adaptive settings, typically require sufficient visitation of the corresponding conditioning events to admit a well-defined long-run interpretation. Dawid's calibration framework makes this reference-class dependence explicit by comparing announced probabilities to realized frequencies along a forecasting sequence \citep{Dawid1982}, while Foster and Vohra show that achieving asymptotic calibration against arbitrary sequences requires an explicit online forecasting protocol, including randomization in the adversarial setting \citep{FosterVohra1998}. In the terminology of this paper, both results treat calibration as a property of a joint forecasting--outcome regime, not as a purely internal property of an isolated probability assignment.
\end{fstatement}

These notions are distinct frequency-oriented meanings or targets. Pathwise stabilization \hyperref[f:F1]{(\ref*{f:F1})} concerns almost-sure
convergence of empirical averages under a fixed law. Procedure-level guarantees \hyperref[f:F2]{(\ref*{f:F2})} are
repeated-sampling or time-uniform statements about an entire adaptive mechanism, often proved via nonnegative
supermartingales and maximal inequalities. Calibration \hyperref[f:F3]{(\ref*{f:F3})} is a prequential reliability
criterion that compares forecasts to realized frequencies along selected subsequences. Internal coherence supplies filtration-level meaning under $\mathbb Q$. Evaluation-law correctness supplies baseline martingale conclusions under $\mathbb P$, including fixed-bin calibration conditional on infinite visitation, but does not by itself supply unconditional stabilization, rich-class calibration, or procedure-level error control. Each frequency claim therefore requires a declared repetition mechanism and the target-specific invariance or drift control that underwrites it.

\subsection{Reference structures and the Principle}
\label{sec2.4-reference-structures}

A \emph{reference structure} is a regime-level specification comprising (i) what constitutes repetition and what is held fixed, including any reference conditioning information, and (ii) the law- or protocol-level invariance, ergodicity, randomization, or controlled-drift assumptions that justify the stated target.
This encompasses exact symmetries (e.g., permutation invariance) and controlled deviations from symmetry (e.g.,
bounded or convergent drift).

In adaptive deployments, it is often helpful to read a reference structure as a \emph{protocol statement}. It declares
which parts of the environment or data generator are treated as fixed across replications, what the forecaster or policy
is allowed to adapt to, whether the procedure is randomized, and which filtration is taken as the information history.
Stated this way, the reference structure becomes operational and auditable, linking the conceptual map to the language of
online learning and anytime-valid inference.

A compact summary of the Principle is:
\begin{equation}
\label{eq:sep-summary}
\begin{aligned}
&\underbrace{\textbf{Conditional-mean correctness}}_{\text{evaluation law}}
\;+\;
\underbrace{\textbf{Target-specific stability}}_{\text{declared repetition regime}}
\\[-0.2em]
&\hspace{3.1cm}\Longrightarrow
\underbrace{\textbf{Specified frequency guarantee}}_{\text{stabilization, calibration, or error control}}.
\end{aligned}
\end{equation}

Internal coherence under $\mathbb Q$ remains the interpretation layer preceding this implication. When $\mathbb Q=\mathbb P$ for the forecast target, coherence also supplies the displayed evaluation-law premise. The remainder of the paper studies the limitations and constructive uses of this implication. Exchangeability yields a transparent alignment between conditionally correct forecasts and empirical frequencies via de Finetti's representation \citep{deFinetti1937,Aldous1985}; beyond it, stabilization reduces to convergence of an appropriate reference drift.

\autoref{tab:ref-structures} in \hyperref[appendix-D1]{Appendix D.1} lists representative reference structures and the frequency
meanings they support. The point is not that any single structure is universally appropriate, but that frequency
claims are \emph{regime-relative}, as they must be interpreted relative to what is held fixed and what stability is
assumed.

The subsequent sections develop this framework. \autoref{sec3-coherence-forecast} formalizes the coherence layer and
its decision-theoretic interpretation. \autoref{sec4-exchangeability} develops exchangeability as a canonical
reference structure that aligns governing-law conditionally correct forecasts and frequency. \autoref{sec5-separations} establishes the two separation results.
\autoref{sec6-beyond-exch} then develops a unified stabilization framework in which reference structures enter through
a drift--innovation decomposition, and \autoref{sec7-guidance} translates the resulting distinctions into guidance for
adaptive and AI-mediated deployments.

\section{Coherence and Probability as Conditional Forecasting and Decision Support}
\label{sec3-coherence-forecast}

\autoref{sec2-map} distinguished internal coherence under an assessment law from empirical frequency meanings under a governing evaluation law. Once an information structure and assessment law are fixed, conditional probability supports forecasting and, together with a stated loss or utility, decision making. Coherence remains well defined when stationarity or exchangeability fails, but it neither verifies the assessment law against deployment nor supplies a repetition mechanism.

Although coherence can be formulated under finite additivity, the martingale arguments and stabilization results developed later in the paper are most naturally stated under countable additivity. We therefore present the coherence layer in the standard measure-theoretic framework.

\subsection{Assessments, conditioning, and coherent updating}
\label{sec3.1-prices}

We interpret an integrable random variable $Y\in L^1(\mathbb Q)$ as a contingent payoff. Work on a filtered assessment space $(\Omega,\mathcal F,(\mathcal F_t)_{t\ge0},\mathbb Q)$, where, as in \autoref{sec2.1-objects}, $\mathcal F_t$ represents the information available when an assessment is issued. In adaptive regimes this information may depend on past observations and interventions.

A filtration-level probabilistic assessment assigns to each such payoff a valuation that is $\mathcal F_t$-measurable and interpreted as a forecast or decision-relevant price conditional on $\mathcal F_t$. 
This interpretation is purely filtration-relative and does not presuppose independence, stationarity, or any regime-level invariance of the realized data stream. 
Under countable additivity, the canonical coherent valuation at time $t$ is conditional expectation under the assessment law:
\begin{equation}
\label{eq:price}
\pi_t(Y)
\;:=\;
\mathbb E_{\mathbb Q}\!\left[\,Y \mid \mathcal F_t\,\right].
\end{equation}

A standard fixed-information-set justification is the absence of a nonzero one-sided payoff. If gambles are priced by $\mathbb E_{\mathbb Q}[\cdot\mid\mathcal G]$, no finite collection of bounded $\mathcal G$-measurable stakes produces a trader payoff that is almost surely one-sided and strictly nonzero with positive probability. This is an internal statement under $\mathbb Q$, not a claim about regularity under $\mathbb P$. Cross-time consistency additionally relies on a coherent family of conditional assessments satisfying the tower property; the fixed-$\mathcal G$ result below does not itself prove a multi-period trading theorem \citep{Savage1954,deFinetti1975}.

This perspective is particularly natural in adaptive environments because the information set
$\mathcal{G}$ can be chosen to reflect the actual data-generating pipeline.  In closed-loop systems,
what the researcher observes is not merely a sequence of outcomes, but a sequence of decisions
and outcomes produced under a policy. The declared information set may therefore be the
$\sigma$-algebra generated by that history; it may include realized actions, covariates, model
outputs, stopping decisions, and any other variables that were available when forecasts were
issued.  Conditioning on this richer $\mathcal{G}$ is not a technical embellishment but the mathematical expression of the interpretive claim that probabilities are forecasts given the information actually used by the system.

To illustrate, consider a company running an internal prediction market for daily deployment success. Let $X_t\in\{0,1\}$ indicate success and let $\mathcal F_t$ contain deployments, code changes, and market information through time $t$. Under an assessment law $\mathbb Q$, a coherent price for the next success is
\[
p_t:=\mathbb Q(X_{t+1}=1\mid\mathcal F_t).
\]
A failure may trigger a hotfix, so the sequence need not be stationary. At each fixed information set, pricing by $p_t$ excludes a nonzero one-sided payoff of the form in \autoref{prop:coherence-decision}. Whether these prices are conditionally correct under the deployment law $\mathbb P$, and whether deployment frequencies stabilize, are separate questions.

Including a stopping rule in the information set similarly preserves this conditional-pricing interpretation, but it does not by itself establish a repeated-sampling guarantee; see \autoref{ex:opt-stopping} in Appendix C.

These examples highlight a core implication. Coherence can be maintained under $\mathbb Q$ even when the assessed process is adaptive. It means that a collection of conditional assessments respects one common probability law and the declared information sets. It does not establish that $\mathbb Q$ equals the deployment law $\mathbb P$, identify a stable long-run target, or justify treating empirical averages as regime-free frequencies.

\subsection{Decision-theoretic meaning: coherent probabilities as optimal forecasts}
\label{sec3.2-forecast}

Coherence also admits a decision-theoretic interpretation under the assessment law. Conditional probabilities are Bayes-optimal forecasts for proper scoring rules, with uniqueness for strictly proper rules \citep{GneitingRaftery2007}. The simplest instance is the strictly proper Brier (squared-error) score for binary events \citep{Brier1950}.

To keep the coherence layer conceptually anchored without proliferating propositions, we package the two standard facts
we use repeatedly into a single statement: (i) conditional expectation avoids sure loss against $\mathcal G$-measurable
trading, and (ii) conditional probability minimizes conditional Brier risk.

\begin{proposition}[Conditional no-arbitrage and Brier-optimal forecasting]
\label{prop:coherence-decision}
\label{prop:opt-forecast}
Let $(\Omega,\mathcal F,\mathbb Q)$ be a probability space and $\mathcal G\subseteq\mathcal F$.
\begin{enumerate}
\item \textnormal{\textbf{(No nonzero one-sided payoff at a fixed information set).}} For integrable gambles $Y_1,\ldots,Y_m\in L^1(\mathbb Q)$ and bounded $\mathcal G$-measurable stakes $H_1,\ldots,H_m$, define the trader's payoff
\[
S := \sum_{i=1}^m H_i\bigl(Y_i-\mathbb{E}_{\mathbb Q}[Y_i\mid \mathcal{G}]\bigr).
\]
Neither $S\ge0$ $\mathbb Q$-a.s. with $\mathbb Q(S>0)>0$ nor $S\le0$ $\mathbb Q$-a.s. with $\mathbb Q(S<0)>0$ is possible.

\item \textnormal{\textbf{(Bayes-optimal forecasting under Brier loss).}} For any event $A\in\mathcal F$, the conditional probability
$\mathbb Q(A\mid\mathcal G)$ is the $\mathbb Q$-a.s.\ unique $\mathcal G$-measurable minimizer of conditional expected Brier loss:
\[
\mathbb{Q}(A\mid\mathcal{G})
\in
\arg\min_{q\in L^0(\mathcal{G};[0,1])}\;
\mathbb{E}_{\mathbb Q}\bigl[(\mathbf{1}_A-q)^2\mid\mathcal{G}\bigr].
\]
\end{enumerate}
\end{proposition}

\noindent\textit{Proof.} Standard; see \hyperref[appendix-A1]{Appendix A.1}.

\medskip

\autoref{prop:coherence-decision} formalizes the interpretive role of probability in the coherence layer: probabilities
are coherent conditional forecasts and coherent updates relative to a specified information set.  This internal
meaning is available regardless of whether any long-run frequency stabilization holds.

This proposition provides an operational interpretation of coherence under $\mathbb Q$: conditional probabilities are Brier-optimal forecasts relative to a specified information set. A probability forecast is not itself a decision rule. Given an action space $\mathcal A$ and loss $\ell$, it supports a Bayes action, when a minimizer exists, through
\[
a_t\in\arg\min_{a\in\mathcal A}\mathbb E_{\mathbb Q}[\ell(a,Y)\mid\mathcal F_t].
\]
The action space and loss or utility are indispensable parts of the decision problem. This interpretation does not presuppose that the assessed data
stream is exchangeable, stationary, or even stable in the long run.  In particular, the forecasting
meaning of probability survives in environments where the data-generating mechanism adapts,
drifts, or responds strategically to past predictions.

The structural point is that coherence operates at the level of the conditional assessment law. For a
given information set $\mathcal{G}$, the forecast $\mathbb{Q}(A\mid \mathcal{G})$ is defined and
evaluated relative to $\mathcal{G}$ alone.  Whether or not empirical averages stabilize, whether or
not repeated-sampling guarantees hold under some regime, and whether or not the system is
stationary are logically separate questions.  Coherence answers the question: ``\textit{Given the
information actually used to make this prediction, is the forecast internally consistent and
decision-theoretically optimal}?''  It does not answer the question: ``\textit{Will long-run frequencies
match these probabilities under some repetition mechanism}?'' This distinction becomes particularly transparent in adaptive data-generating regimes. 

Consider an online platform that, at each time $t$, selects an action $A_t \in \{0,1\}$ (e.g., which item to display to a user) based on past information $\mathcal{F}_{t-1}$ and then observes an outcome $X_t \in \{0,1\}$ (e.g., whether the user clicks on the recommended item). Let
\[
\mathcal{F}_t := \sigma(A_1,X_1,\ldots,A_t,X_t)
\]
denote the realized history. The governing-law conditionally correct forecast for the next outcome is the conditional
probability of a user clicking on a recommended item,
\[
p_t^* := \mathbb{P}(X_{t+1}=1 \mid \mathcal{F}_t),
\]
and the platform uses this forecast to select an action $A_{t+1}$. Suppose that user behavior gradually changes in response to prior recommendations, so that the conditional distribution of $X_{t+1}$ given $\mathcal{F}_t$ evolves over time in a
path-dependent way. The sequence $(X_t)$ need not be stationary, exchangeable, or convergent in frequency. 

At each time $t$, $p_t^*$ is the unique minimizer of conditional expected Brier loss given $\mathcal F_t$ under the displayed law. This establishes neither global optimality of the action policy nor convergence of realized loss to zero. Whether average loss stabilizes, and to what Bayes-risk level, is a separate regime-level question.

This separation extends to parameter inference, such as in adaptive clinical trials where coherent posteriors do not guarantee frequentist error control (detailed in \autoref{ex:adp-expr}).

These examples clarify a structural asymmetry that underlies the Principle. The belief role of probability is intrinsically filtration-relative, as it is defined with respect to the information actually used in forecasting and decision making. As long as conditional expectations are well defined, coherence ensures that probabilities function as optimal forecasts under proper
loss functions, regardless of the stability or instability of the data-generating process.
By contrast, frequency roles are inherently regime-relative because they require specifying how the process is to be repeated and which structural features are assumed invariant across replications.

In modern adaptive environments, where policies evolve and outcomes respond to prior predictions,
this distinction is not philosophical but operational.  Coherence supplies a stable interpretive
anchor for probabilistic reasoning at each time point, even when the global system lacks
long-run regularity.  Whether and how those coherent forecasts connect to long-run frequencies is
a separate question, addressed only after a reference structure has been made explicit.

\subsection{Evaluation-law conditional correctness and validity in the large}
\label{sec3.3-sequential}

Return to the adaptive-regime setup of \autoref{sec2.1-objects}.  Let $(Z_t)_{t\ge1}$ denote the observed stream and
$\mathcal F_t=\sigma(Z_1,\ldots,Z_t)$ the induced filtration.  Let $(X_t)_{t\ge1}$ be an adapted outcome process with
$0\le X_t\le 1$ and empirical mean $\bar X_n$ as in \eqref{eq:emp-mean}.

The deployment benchmark from \eqref{eq:eval-correctness} is conditional-mean correctness under the governing evaluation law $\mathbb P$:
\[
p_t=\mathbb E_{\mathbb P}[X_{t+1}\mid\mathcal F_t],
\]
equivalently,
\begin{equation}
\label{eq:one-step}
p_{t-1}:=\mathbb E_{\mathbb P}[X_t\mid\mathcal F_{t-1}],\qquad t\ge1.
\end{equation}
We do not assume that deployed AI systems satisfy this identity. Internal coherence under an assessment law $\mathbb Q$ implies it only when the relevant one-step conditional means under $\mathbb Q$ and $\mathbb P$ agree.

Define innovations
\[
d_t:=X_t-p_{t-1},\qquad t\ge1.
\]
Then $\mathbb E_{\mathbb P}[d_t\mid\mathcal F_{t-1}]=0$ $\mathbb P$-a.s., so $(d_t)$ is a $\mathbb P$-martingale-difference sequence and
$\sum_{t=1}^n d_t$ is a martingale \citep{Doob1949}.  A martingale strong law for bounded martingale differences yields
the baseline unweighted martingale conclusion implied by evaluation-law conditional correctness:
\begin{equation}
\label{eq:vitl}
\frac{1}{n}\sum_{t=1}^n (X_t-p_{t-1})\ \xrightarrow{\mathbb P\text{-a.s.}}\ 0,
\qquad\text{equivalently}\qquad
\bar X_n-\frac{1}{n}\sum_{t=1}^n p_{t-1}\ \xrightarrow{\mathbb P\text{-a.s.}}\ 0.
\end{equation}
Equation \eqref{eq:vitl} states validity in the large: forecast errors cancel on average along a $\mathbb P$-typical trajectory. It is the baseline unweighted consequence of the displayed $\mathbb P$-conditional-mean identity, not a consequence of internal $\mathbb Q$-coherence alone.

\begin{remark}[Validity in the large versus calibration]
\label{rem:vitl-vs-calibration}
Equation \eqref{eq:vitl} is generally insufficient for unconditional or rich-class calibration in the sense of
\hyperref[f:F3]{\textbf{(\ref*{f:F3})}}.  In particular, \eqref{eq:vitl} controls only the single global discrepancy
$\frac1n\sum_{t=0}^{n-1}(X_{t+1}-p_t)$; it does not control discrepancies restricted to data-dependent subsets of times,
such as those defined by binning forecasts (e.g., $\{t: p_t\in I\}$ for an interval $I$) or, more generally, by a rich
class of $\mathcal F_t$-measurable checking rules.  Obtaining binwise (or conditional) calibration typically requires
additional structure limiting either the complexity of the checking class or the adversarial adaptivity of the joint
process generating $(p_t,X_{t+1})$; see, for example, the prequential formulation of calibration and its refinements
in \citet{Dawid1982} and the online-calibration results of \citet{FosterVohra1998}.  In the language of this paper,
conditional-mean correctness also yields vanishing calibration gap for each fixed, predeclared Borel bin on the event that the bin is visited infinitely often (\autoref{prop:fixed-bin-calibration}). It does not ensure that this event occurs, provide an unconditional nonvacuous guarantee, or establish uniform calibration over a rich or data-dependent checking class.
\end{remark}

\subsection{What coherence and conditional correctness do \emph{not} deliver}
\label{sec3.4-not-frequency}

Internal coherence under $\mathbb Q$ supplies conditional-assessment meaning. Evaluation-law conditional correctness under $\mathbb P$ separately supplies \eqref{eq:vitl} and the limited fixed-bin result just described. Neither condition supplies all regime-level content needed for the frequency meanings in \autoref{sec2.3-frequency}:
\begin{itemize}
\item Conditional-mean correctness does not imply pathwise stabilization of $\bar X_n$ nor even existence of a limit
(\hyperref[f:F1]{\ref*{f:F1}}).
\item Internal coherence does not yield repeated-sampling guarantees for procedures without an explicit replication mechanism and a valid procedure under that mechanism
(\hyperref[f:F2]{\ref*{f:F2}}).
\item Conditional-mean correctness does not guarantee visitation, unconditional calibration, or uniform calibration over rich checking classes
(\hyperref[f:F3]{\ref*{f:F3}}).
\end{itemize}

In the language of the Principle, coherence provides filtration-level meaning under its assessment law. Empirical guarantees under the evaluation law require the premises appropriate to the stated target. The next section begins with exchangeability as a canonical regime in which conditionally correct forecasts and long-run frequencies align through a latent representation.

\section{Exchangeability as a Canonical Reference Structure}
\label{sec4-exchangeability}

Exchangeability is the cleanest setting in which the Principle becomes transparent.  It is a regime-level
symmetry assumption that makes ``repetition'' explicit.  Under this reference
structure, one-step forecasts that are conditionally correct under the governing exchangeable law and empirical
frequencies stabilize to the same latent quantity. This section develops exchangeability as the canonical alignment case in \autoref{tab:ref-structures}.

We focus on the binary case for clarity and because it connects directly to the stabilization meaning
\hyperref[f:F1]{(\ref*{f:F1})}.  Extensions to general Polish spaces are standard \citep{Aldous1985}.

\subsection{Exchangeability as permutation invariance}
\label{sec4.1-exch-def}

Let $(X_t)_{t\ge1}$ be a $\{0,1\}$-valued process with natural filtration $\mathcal F_t=\sigma(X_1,\ldots,X_t)$.  The
sequence is \emph{exchangeable} if for every $n\ge1$ and every permutation $\pi$ of $\{1,\ldots,n\}$,
\[
(X_1,\ldots,X_n)\ \overset{d}{=}\ (X_{\pi(1)},\ldots,X_{\pi(n)}).
\]
Exchangeability is a reference structure in the sense of \autoref{sec2.4-reference-structures}. It specifies what is
held fixed under repetition (the joint law up to reordering) and supplies a symmetry that legitimizes long-run averages
as frequency objects.  In adaptive regimes, exchangeability is typically violated. Its role here is therefore diagnostic, as it isolates what additional content beyond coherence is needed to recover frequency meaning.

\subsection{Latent representation (de Finetti) and its interpretive role}
\label{sec4.2-definetti}

The key fact is de Finetti’s representation theorem (\autoref{them:definetti}): an exchangeable Bernoulli sequence is a mixture of i.i.d.\
Bernoulli laws \citep{deFinetti1937,deFinetti1975,Aldous1985}.  Concretely, if $(X_t)$ is exchangeable, then there
exists a latent variable $\Theta\in[0,1]$ such that, conditional on $\Theta$, the sequence is i.i.d.\
$\mathrm{Bernoulli}(\Theta)$.  Equivalently, the exchangeable law is a mixture of product measures indexed by
$\theta\in[0,1]$. Since this result is classical, we defer a formal statement to \hyperref[appendix-A13]{Appendix A.13}.

In the language of \autoref{sec2-map}, $\Theta$ is the regime-level object induced by the symmetry; it is the latent quantity to which long-run empirical frequencies and governing-law conditionally correct one-step forecasts converge under exchangeability.

\subsection{Convergence of Governing-Law Forecasts and Long-run Frequency}
\label{sec4.3-alignment}

Define the one-step forecast that is conditionally correct under the governing exchangeable law:
\[
p_t:=\mathbb P(X_{t+1}=1\mid\mathcal F_t)=\mathbb E[X_{t+1}\mid\mathcal F_t],\qquad t\ge0.
\]
Under de Finetti’s representation, $p_t=\mathbb E[\Theta\mid\mathcal F_t]$, so $(p_t)$ is a bounded martingale.  In
parallel, conditional on $\Theta$ the strong law gives $\bar X_n\to\Theta$ almost surely.  Thus exchangeability makes
precise a canonical bridge from evaluation-law conditional correctness to frequency stabilization. If the assessment and governing laws agree on the forecast target, the same forecasts are also internally coherent.

\begin{proposition}[Exchangeability aligns forecasts and frequency]
\label{prop:exch-align}
Let $(X_t)_{t\ge1}$ be an exchangeable $\{0,1\}$-valued sequence with natural filtration
$\mathcal F_t=\sigma(X_1,\ldots,X_t)$.  Let $\Theta\in[0,1]$ be the latent variable from de Finetti’s representation
(\autoref{them:definetti}), and define $p_t:=\mathbb P(X_{t+1}=1\mid\mathcal F_t)$.
Then:
\begin{enumerate}
\item $ (p_t,\mathcal F_t)_{t\ge0}$ is a bounded martingale and $p_t\to\Theta$ almost surely.
\item The empirical mean $\bar X_n:=\frac1n\sum_{t=1}^n X_t$ converges almost surely to the same limit:
$\bar X_n\to\Theta$.
\end{enumerate}
\end{proposition}

\noindent\textit{Proof.} Standard; see \hyperref[appendix-A2]{Appendix A.2}.  Part (1) follows from $p_t=\mathbb E[\Theta\mid\mathcal F_t]$
and martingale convergence; part (2) follows from the strong law conditional on $\Theta$.

\medskip

\autoref{prop:exch-align} is the canonical alignment case promised by \eqref{eq:sep-summary}. Conditional correctness under the
governing exchangeable law yields a clear frequency meaning (\hyperref[f:F1]{\ref*{f:F1}}), with a latent limit that also
governs stabilized forecasts. When the assessment law agrees with that governing law, these forecasts are internally coherent as well. This alignment is special because the exchangeability symmetry is typically absent in adaptive regimes. Even when empirical averages stabilize, their limits need not identify the underlying law or justify stronger frequency notions such as calibration. These limitations motivate the separation results established next in \autoref{sec5-separations}.

\section{Two Separation Results}
\label{sec5-separations}

\autoref{sec4-exchangeability} presented exchangeability as a canonical reference structure under which governing-law
conditional forecasts and frequency stabilization align.  The Principle, however, is fundamentally a statement about what
happens without such regime structure.  This section establishes two complementary separation results that
clarify the limits of frequency summaries and coherence alone.

Both results are logical rather than empirical.  They do not claim that stabilization or identification is impossible
in applications.  Instead, they formalize two boundary facts: (i) partial long-run-frequency information, even when a limit exists, can be too coarse to identify the underlying law; (ii) even conditional-mean correctness under the governing law does not force empirical frequencies to converge.

Together, these results motivate the need for the regime-level assumptions and drift conditions developed in
\autoref{sec6-beyond-exch}.

\subsection{Limits and partial limiting-frequency summaries need not identify the law}
\label{sec5.1-nonid}

Work on the path space $\mathcal X:=\{0,1\}^{\mathbb N}$ with its product $\sigma$-algebra $\mathcal B(\mathcal X)$. For a process $(X_t)_{t\ge1}$, define $\bar X_n$ as in \eqref{eq:emp-mean} and the globally measurable path functional
\[
L(\omega):=
\begin{cases}
\displaystyle\lim_{n\to\infty}\bar X_n(\omega),&\omega\in D,\\
0,&\omega\notin D.
\end{cases}
\]

The random variable $L$ is a path functional whose distribution is determined by the path law. Within the binary exchangeable class, $L=\Theta$ almost surely and $\mathcal L_{\mathbb P}(L)=\mu$. Because the de~Finetti mixing measure uniquely determines the exchangeable law, the full distribution of $L$ identifies that law \citep{deFinetti1937,kallenberg2005probabilistic}. The non-identification result below therefore concerns a proper coarsening of $\mathcal L_{\mathbb P}(L)$, such as the single atom probability $\mathbb P(L=p)$. A single realized value of $L$ likewise supplies only one path-level observation, not the across-replication distribution $\mu$. Under stationarity, Birkhoff's theorem yields $\bar X_n\to\mathbb E[X_1\mid\mathcal I]$ almost surely; one time average generally reveals only the realized value of this functional, not the invariant component or full stationary law \citep{Birkhoff1931,Walters1982,Krengel1985}.

In classical i.i.d.\ settings, $L$ is often interpreted as ``the'' underlying probability of success.  The conceptual
map of \autoref{sec2-map} emphasizes that such an interpretation already encodes a reference structure: i.i.d.\
sampling supplies a repetition mechanism under which $L$ is meaningful and informative.  Outside such a structure, $L$
captures only a single asymptotic feature of the law on $\mathcal X$ and need not determine the full predictive
content of the system.

The next proposition makes this precise in the most favorable case of exchangeability.  Even within the exchangeable
class, coarse frequency information does not identify the law.

\begin{proposition}[A proper coarsening of the limiting-frequency law does not identify the exchangeable law]
\label{prop:freq-not-law}
\leavevmode
\begin{enumerate}
\item \textnormal{\textbf{(Exchangeability implies stabilization only through a latent measure).}}
If $\mathbb P$ is exchangeable on $(\mathcal X,\mathcal B(\mathcal X))$, then there exists a $[0,1]$-valued random
variable $\Theta$ such that
\[
\bar X_n \xrightarrow{\mathrm{a.s.}} \Theta,
\]
and the law of $\Theta$ is the de Finetti mixing measure $\mu$ on $[0,1]$.  In particular, for each $p\in[0,1]$,
\begin{equation}
\label{eq:atom-mass}
\mathbb P(L=p)=\mu(\{p\}).
\end{equation}

\item \textnormal{\textbf{(Non-identifiability from coarse frequency information).}}
Fix any $p\in(0,1)$ and any $\alpha\in(0,1)$.  There exist distinct exchangeable probability laws
$\mathbb P_1\neq \mathbb P_2$ on $(\mathcal X,\mathcal B(\mathcal X))$ such that
\[
\mathbb P_1(L=p)=\mathbb P_2(L=p)=\alpha.
\]
\end{enumerate}
\end{proposition}

\noindent\textit{Proof.} See \hyperref[appendix-A3]{Appendix A.3}.

\medskip

\autoref{prop:freq-not-law} does not deny identification from the full law of $L$ within the binary exchangeable class. It shows that a proper summary, here $\mathbb P(L=p)$, can be shared by distinct exchangeable laws. A single realized limit is likewise only one observation of $L$, not its across-replication law. Identification from such information requires richer across-replication data or additional identifying restrictions.

For concrete illustrations of this non-identification phenomenon in adaptive systems, see \autoref{ex:freq-not-law}, \autoref{ex:rec-alg}, and \autoref{ex:adp-clin}.

The interpretive content of a realized $L$, or partial information about its distribution, depends on the declared regime specifying what is held fixed, how policies evolve, and what invariance or controlled drift is assumed. Behavior under alternative policies additionally requires a causal or structural intervention model and its identification assumptions. It is therefore not justified to treat one limit or a partial summary as a stand-alone definition of the governing law.

This establishes the first separation emphasized in the paper. A proper coarsening of a limiting-frequency law need not recover the governing path law. In feedback-driven environments, such summaries reflect the joint dynamics of outcomes and policy updates.

\subsection{Conditional-mean correctness does not imply stabilization}
\label{sec5.2-nonconv}

\autoref{sec3-coherence-forecast} showed that conditional-mean correctness under the governing law yields a martingale structure and validity in the large \eqref{eq:vitl}. One might hope that this condition would also enforce stabilization of empirical
frequencies, i.e.\ almost sure existence of $\lim_{n\to\infty}\bar X_n$ (meaning
\hyperref[f:F1]{(\ref*{f:F1})}). The following proposition shows that it does not.

\begin{proposition}[Conditional-mean correctness does not imply stabilization]
\label{prop:coh-not-stable}
There exists a non-degenerate, countably additive law $\mathbb R$ on the binary path space $\mathcal X$ such that, with $\mathcal F_t=\sigma(X_1,\ldots,X_t)$ and
\[
p_{t-1}:=\mathbb E_{\mathbb R}[X_t\mid\mathcal F_{t-1}],
\]
the forecasts are conditionally correct under $\mathbb R$ and
\[
\frac1n\sum_{t=1}^n(X_t-p_{t-1})\xrightarrow{\mathbb R\text{-a.s.}}0,
\qquad
\mathbb R(D)=0.
\]
Thus conditional-mean correctness and validity in the large do not imply stabilization of $\bar X_n$.
\end{proposition}

\noindent\textit{Proof.} See \hyperref[appendix-A4]{Appendix A.4}.

\begin{exmpm}[RL feedback with conditional correctness but without stabilization]
\label{ex:coh-not-stable}    
We include a stylized reinforcement-learning (RL) simulation in a
one-state, two-action environment. Rewards are conditionally Bernoulli given the predictable action sequence, and the policy is updated from each episode-average reward. With $\mathcal F_t=\sigma(A_1,X_1,\ldots,A_t,X_t)$ and $p_{t-1}=\mathbb E[X_t\mid\mathcal F_{t-1}]=r_{A_t}$, the innovations are martingale differences. A Hoeffding--Borel--Cantelli argument establishes eventual alternating actions and two distinct end-of-episode limits, so $\bar X_n$ fails to converge almost surely. \autoref{fig:rl-coh-not-stable} is a finite illustration; the infinite-horizon result is proved in \hyperref[app:sim-rl-coh-not-stable]{Appendix B}.
\end{exmpm}

\autoref{prop:coh-not-stable} is a generic nonstationary existence result: its block variables are exogenous, so the proposition itself does not model endogenous closed-loop learning. Appendix B gives a separate feedback-driven construction under explicit reward-generation and policy-update assumptions. Both show that conditional-mean correctness does not control convergence of the predictable drift. Neither result suggests that nonconvergence is typical.

The result is consistent with \eqref{eq:vitl}: martingale innovations can average to zero while the predictable mean continues to oscillate. Stabilization therefore remains a distinct, regime-dependent target.

Accordingly, evaluation-law conditional correctness secures innovation cancellation but not convergence of the policy--environment drift. Establishing stabilization requires explicit assumptions about policy evolution, held-fixed components, and invariance or controlled drift.

\subsection{Implication for the Principle and transition to stabilization theory}
\label{sec5.3-transition}

The two propositions above formalize the limiting implications of the Principle. 
\autoref{prop:freq-not-law} establishes non-identification from a proper coarsening of the limiting-frequency law. \autoref{prop:coh-not-stable} shows that evaluation-law conditional correctness does not guarantee long-run stabilization.

We next develop the constructive side of the Principle by introducing a unified stabilization framework in which empirical averages are decomposed into a predictable ``reference drift'' component and a martingale innovation component. 
Under explicit regime-level assumptions that control the reference drift, such as exchangeability, stationarity or ergodicity, a fixed policy--environment pair, or design-based randomization, this framework yields stabilization results and clarifies which frequency interpretations are warranted in adaptive regimes.

\section{A Unified Stabilization Framework Beyond Exchangeability}
\label{sec6-beyond-exch}

The purpose of this section is to develop the complementary constructive direction of the Principle.  We do not introduce new limit theorems.  Rather, we organize standard stabilization results under a single template that makes explicit (i) which part of a frequency guarantee is a purely filtration-level consequence of conditional expectation, and (ii) where regime-level assumptions (reference structures) enter to control systematic drift.

\subsection{Reference--drift decomposition and a unified stabilization theorem}
\label{sec6.1-unified}

Let $(Z_t)_{t\ge1}$ be an observed stream and let $X_t:=g(Z_t)$ be a bounded statistic of interest, such as an error indicator, click-through event, or bounded reward. Let $(\mathcal F_t)_{t\ge0}$ be the filtration generated by the observed history. A full reference structure $\mathfrak R$ specifies the repetition protocol, held-fixed variables, governing law, and target-specific invariance or drift assumptions.

To encode held-fixed variables or analyst-side reference conditioning information, such as latent regimes or fixed potential outcomes, introduce a reference $\sigma$-algebra $\mathcal R\subseteq\mathcal F$
and consider the augmented information sets
\[
\mathcal H_t \;:=\; \mathcal R\vee \mathcal F_t,\qquad t\ge0.
\]
The $\sigma$-algebra $\mathcal R$ is one component of $\mathfrak R$; stationarity, ergodicity, policy-kernel restrictions, randomization, and drift convergence are properties of the governing law or protocol rather than of $\mathcal R$ alone. Define the \emph{reference drift} using the reference conditioning information in the probability analysis:
\begin{equation}
\label{eq:ref-drift-sec6}
m_t \;:=\; \mathbb E_{\mathbb P}[X_t\mid \mathcal H_{t-1}],\qquad t\ge1.
\end{equation}
Define the innovation (martingale difference)
\[
d_t \;:=\; X_t-m_t,\qquad t\ge1.
\]
By construction, $\mathbb E[d_t\mid \mathcal H_{t-1}]=0$ almost surely, so $(d_t)$ is a martingale difference sequence
with respect to $(\mathcal H_t)$.  Averaging yields the basic reference--drift decomposition
\begin{equation}
\label{eq:drift-innovation-sec6}
\bar X_n \;=\; \frac{1}{n}\sum_{t=1}^n m_t \;+\; \frac{1}{n}\sum_{t=1}^n d_t.
\end{equation}

\begin{remark}[Connection to Doob's predictable--martingale decomposition]
\label{rem:doob-connection}
The split \eqref{eq:drift-innovation-sec6} is in the spirit of Doob's predictable--martingale decomposition. For the
partial sums $S_n:=\sum_{t=1}^n X_t$, one may write $S_n=A_n+M_n$ where $(A_n)$ is predictable and $(M_n)$ is a
martingale. We work directly with empirical averages and make the regime dependence explicit through the
reference $\sigma$-algebra $\mathcal R$, so that the predictable component is the reference drift $m_t=\mathbb E[X_t\mid\mathcal H_{t-1}]$.
This specialization isolates the precise point at which reference structures enter stabilization and time-uniform guarantees.
\end{remark}

The second term in \eqref{eq:drift-innovation-sec6} is the familiar ``innovation cancellation'' mechanism underlying validity in the large \eqref{eq:vitl}. The regime content is therefore carried by the first term, which depends on the chosen reference structure. Accordingly, stabilization of $\bar X_n$ reduces to stabilization of the reference drift averages.

\begin{proposition}[Unified stabilization via a reference drift]
\label{prop:unified-stabilization}
Let $(\Omega,\mathcal F,\mathbb P)$ be a probability space, $(\mathcal F_t)_{t\ge0}$ a filtration, and let
$(X_t)_{t\ge1}$ be adapted with
\begin{equation}
\label{eq:bounded-X-sec6}
0\le X_t\le 1\quad\text{a.s.\ for all }t\ge1.
\end{equation}
Let $\mathcal R\subseteq\mathcal F$ be a reference $\sigma$-algebra and define $\mathcal H_t=\mathcal R\vee\mathcal F_t$
and $m_t=\mathbb E[X_t\mid\mathcal H_{t-1}]$ as in \eqref{eq:ref-drift-sec6}.
If there exists a random variable $\theta$ such that
\begin{equation}
\label{eq:cesaro-drift-sec6}
\frac{1}{n}\sum_{t=1}^n m_t \xrightarrow{\mathrm{a.s.}} \theta,
\end{equation}
then the empirical mean stabilizes to the same limit:
\begin{equation}
\label{eq:barX-to-theta-sec6}
\bar X_n:=\frac{1}{n}\sum_{t=1}^n X_t \xrightarrow{\mathrm{a.s.}} \theta.
\end{equation}
Equivalently, 
\[
\bar X_n-\frac1n\sum_{t=1}^n m_t \xrightarrow{\mathrm{a.s.}}0.
\]
\end{proposition}

\noindent\textit{Proof.} See \hyperref[appendix-A5]{Appendix A.5}. (Apply a martingale strong law to the bounded martingale differences
$d_t=X_t-\mathbb E[X_t\mid\mathcal H_{t-1}]$.)

\medskip

\autoref{prop:unified-stabilization} makes the Principle operational at theorem level. It separates innovation cancellation under the governing law from reference-drift stabilization under the declared regime. If $(d_t)$ has martingale-difference structure with bounded increments, then a martingale strong law yields
\[
\frac{1}{n}\sum_{t=1}^n d_t \;\longrightarrow\; 0 \quad \text{a.s.}
\]
This conclusion is law- and filtration-relative. It is robust to adaptivity because it does not invoke stationarity or independence, but it is not implied by internal coherence under a different assessment law.

The second ingredient is \emph{reference-drift stabilization}. Stabilization of empirical averages additionally requires that the predictable component $(m_t)$ satisfy a regime-level convergence condition,
\[
\frac{1}{n}\sum_{t=1}^n m_t \;\longrightarrow\; \theta \quad \text{a.s.}
\]
This is the point at which reference-structure assumptions enter. It is also the ingredient that fails in \autoref{prop:coh-not-stable} and in the reinforcement-learning illustration (\autoref{ex:coh-not-stable}), where innovations cancel as above while the predictable component continues to oscillate under persistent regime switching, thereby preventing stabilization of the empirical mean.

The proposition also clarifies why stabilization need not identify the full law (\autoref{sec5.1-nonid}). Even when $\bar X_n \to \theta$, the limit is a reference-structure-dependent projection of the underlying law, determined through the drift process $(m_t)$ and the particular statistic $(X_t)$ being averaged, rather than a sufficient summary of the joint law governing finite-horizon predictions.

\begin{remark}[Exchangeability as a special case]
\label{rem:exch-spec}
Under exchangeability, one may take $\mathcal R=\sigma(\Theta)$ where $\Theta$ is the de Finetti latent variable
(\autoref{sec4.2-definetti}). Then $m_t=\mathbb E[X_t\mid \Theta,\mathcal F_{t-1}]=\Theta$, so \eqref{eq:cesaro-drift-sec6} holds
trivially with $\theta=\Theta$, recovering $\bar X_n\to\Theta$ (\autoref{prop:exch-align}).
\end{remark}

\begin{exmpm}[Diagnosing non-stabilization in an RL feedback regime via the drift--innovation split]
\label{ex:drift-diagnosis-rl}
Consider the non-stabilizing RL-style construction in \autoref{ex:coh-not-stable}.  Let the realized
history filtration be
\[
\mathcal{F}_t:=\sigma(A_1,X_1,\ldots,A_t,X_t),
\]
and take the reference conditioning $\sigma$-algebra $\mathcal R$ to be trivial, so that
$\mathcal{H}_t=\mathcal{F}_t$.  The unified framework of Section~\ref{sec6.1-unified} decomposes
\[
X_t = m_t + d_t,
\qquad
m_t:=\mathbb{E}[X_t\mid \mathcal{F}_{t-1}],
\qquad
d_t:=X_t-m_t,
\]
where $(d_t)$ is a bounded martingale-difference sequence.

In this construction, the environment is stationary conditional on the action, in the sense that
\[
m_t=\mathbb{E}[X_t\mid \mathcal{F}_{t-1}]
=
\mathbb{E}[X_t\mid A_t]
=
r_{A_t},
\qquad r_{A_t}\in\{r_0,r_1\},
\]
so the entire predictable component of $X_t$ is encoded by the action sequence.  The
innovation term is asymptotically negligible for the empirical average. By bounded martingale-difference cancellation,
\[
\frac{1}{n}\sum_{t=1}^n d_t \xrightarrow{\mathrm{a.s.}} 0.
\]
Therefore, whether the empirical mean stabilizes is determined entirely by the drift average:
\[
\bar X_n
=
\frac{1}{n}\sum_{t=1}^n m_t
+
\frac{1}{n}\sum_{t=1}^n d_t
=
\frac{1}{n}\sum_{t=1}^n r_{A_t}
+
o(1)
\quad\text{a.s.}
\]
\end{exmpm}

\subsection{Filtration-based Stability}
\label{sec6.2-filtration}

In many adaptive deployments, the most natural reference structure is purely sequential in the sense that repetition and comparability are understood relative to the realized information history (actions, covariates, model states, outcomes), without additional latent invariants.  In \autoref{sec6.1-unified}, this corresponds to taking $\mathcal R$ trivial, so that
$\mathcal H_t=\mathcal F_t$ and $m_t=\mathbb E_{\mathbb P}[X_t\mid\mathcal F_{t-1}]$ is the governing-law one-step conditional mean.

A particularly transparent sufficient condition for stabilization is that these predictive means themselves stabilize along the
realized filtration.

\begin{proposition}[Predictive convergence implies stabilization of empirical averages]
\label{prop:filtration-stabilization}
Let $(X_t)_{t\ge1}$ be adapted to $(\mathcal F_t)_{t\ge0}$ with $0\le X_t\le 1$ a.s.\ for all $t$.
Define the governing-law one-step predictive means
\[
p_t := \mathbb E_{\mathbb P}[X_{t+1}\mid \mathcal F_t],\qquad t\ge 0,
\]
and the empirical mean $\bar X_n = \frac1n\sum_{t=1}^n X_t$.
If $p_t \xrightarrow{\mathrm{a.s.}} p_\infty$ for some random variable $p_\infty$, then
\[
\bar X_n \xrightarrow{\mathrm{a.s.}} p_\infty.
\]
\end{proposition}

\noindent\textit{Proof.} See \hyperref[appendix-A6]{Appendix A.6}, where the result follows by applying \autoref{prop:unified-stabilization} with $\mathcal R$ trivial and observing that Ces\`aro averaging preserves almost sure limits. An illustrative example is given in \autoref{ex:freezeout}.

Filtration-based reference structures also support \emph{anytime-valid} guarantees of type
\hyperref[f:F2]{(\ref*{f:F2})}, which remain valid under data-dependent stopping.
A standard route is through nonnegative supermartingales and Ville-type inequalities
\citep{Ville1939,HowardRamdasMcAuliffeSekhon2020}.

\begin{proposition}[Time-uniform control and optional stopping]
\label{prop:ville-inequality}
Let $(M_t)_{t\ge0}$ be a nonnegative supermartingale with respect to $(\mathcal F_t)_{t\ge0}$ such that $M_0=1$.
Then for any $c>0$,
\[
\mathbb P\!\left(\sup_{t\ge0} M_t \ge c\right)\le \frac1c.
\]
Equivalently, for any $\alpha\in(0,1)$,
\[
\mathbb P\!\left(\sup_{t\ge0} M_t \ge \frac1\alpha\right)\le \alpha.
\]
\end{proposition}

\noindent\textit{Proof.} Standard; see \hyperref[appendix-A7]{Appendix A.7}.

\subsection{Shift-based stability}
\label{sec6.3-shift}

Shift-based reference structures formalize stability as invariance under time translation.  They are canonical in time
series and dynamical systems, and they sometimes provide a useful approximation for closed-loop systems operating
within a stable regime window. Here, repetition is encoded not by re-running a policy or design,
but by assuming the law itself is invariant under the shift operator.

\begin{proposition}[Stationarity/ergodicity implies stabilization of time averages]
\label{prop:shift-stabilization}
Let $(X_t)_{t\ge1}$ be a strictly stationary process with $\mathbb E|X_1|<\infty$ (in particular, the bounded case
\eqref{eq:bounded-X-sec6} satisfies this). Let $\mathcal I$ denote the shift-invariant $\sigma$-algebra. Then
\[
\bar X_n \xrightarrow{\mathrm{a.s.}} \mathbb E[X_1\mid \mathcal I].
\]
If the process is ergodic (i.e., $\mathcal I$ is trivial), then $\bar X_n\to \mathbb E[X_1]$ almost surely.
\end{proposition}

\noindent\textit{Proof.} See \hyperref[appendix-A8]{Appendix A.8} (Birkhoff's ergodic theorem; \citealp{Birkhoff1931,Walters1982,Krengel1985}).

\autoref{prop:shift-stabilization} is a shift-symmetry route to \hyperref[f:F1]{(\ref*{f:F1})} stabilization.  It highlights the regime dependence emphasized by the Principle. The limit is not, in general, a universal constant but rather the conditional expectation given the invariant $\sigma$-algebra that encodes the reference structure. In feedback-driven deployments, stationarity can fail due to policy updates or endogenous behavioral responses; in such cases, shift-based stability must be justified explicitly or replaced by an alternative reference structure that controls drift more directly (as in \autoref{sec6.2-filtration} and \autoref{sec6.5-design}). A concrete instance of this regime-dependent stabilization is given in \autoref{ex:nonergodic}.

\subsection{Policy--environment Regimes}
\label{sec6.4-policy-env}

Policy--environment reference structures make explicit that long-run frequency meaning in reinforcement learning is
regime-indexed: repetition corresponds to re-running the same fixed policy and environment kernel. This is the
formal content of statements of the form ``under the fixed pair $(\pi,\mathsf P)$'' in the RL literature
\citep{SuttonBarto2018,LattimoreSzepesvari2020}.

Let $(S_t)_{t\ge0}$ take values in a measurable state space $(\mathsf S,\mathcal S)$ and let $(\mathsf A,\mathcal A)$ be a measurable action space. Let $\pi(da\mid s)$ be a stationary Markov kernel from $(\mathsf S,\mathcal S)$ to $(\mathsf A,\mathcal A)$, and draw each $A_t$ conditionally from $\pi(\cdot\mid S_t)$ given the state and prior trajectory. The environment is a Markov transition/reward kernel from $(\mathsf S\times\mathsf A,\mathcal S\otimes\mathcal A)$ to $(\mathsf S\times[0,1],\mathcal S\otimes\mathcal B([0,1]))$, written
\[
\mathsf P(ds',dx\mid s,a),\qquad 0\le x\le 1,
\]
so that
\[
(S_{t+1},X_{t+1})\mid\{(S_t,A_t)=(s,a)\}
\sim \mathsf P(\cdot,\cdot\mid s,a).
\]
Let
\[
\mathcal F_t:=\sigma(S_0,A_0,X_1,S_1,\ldots,S_t,A_t)
\]
be the natural filtration and define the
one-step mean reward $r(s,a):=\mathbb E[X_{t+1}\mid S_t=s,A_t=a]\in[0,1]$. Assume the induced state--action process $(S_t,A_t)_{t\ge0}$ is a time-homogeneous Markov chain that is positive Harris recurrent and aperiodic, with invariant probability measure $\nu_{\pi,\mathsf P}$ on $\mathsf S\times\mathsf A$.

\begin{proposition}[Ergodic policy--environment regimes imply stabilization of average reward]
\label{prop:policy-stabilization}
Assume that under the fixed pair $(\pi,\mathsf P)$ the state-action process $(S_t,A_t)_{t\ge0}$ is a time-homogeneous
Markov chain that is positive Harris recurrent and aperiodic, with invariant probability measure
$\nu_{\pi,\mathsf P}$ on $\mathsf S\times\mathsf A$.
Then
\[
\bar X_n := \frac1n\sum_{t=1}^n X_t \xrightarrow{\mathrm{a.s.}}
\rho(\pi,\mathsf P)
:= \int_{\mathsf S\times\mathsf A} r(s,a)\,\nu_{\pi,\mathsf P}(ds,da),
\]
where $r(s,a):=\mathbb E[X_{t+1}\mid S_t=s,A_t=a]$.
\end{proposition}

\noindent\textit{Proof.} See \hyperref[appendix-A9]{Appendix A.9}; the positive-Harris ergodic theorem is given, for example, by \citet{MeynTweedie1993}. A concrete illustration is given in \autoref{ex:bandit}.

\subsection{Design-based Regimes}
\label{sec6.5-design}

Design-based reference structures treat the randomization mechanism as the probabilistic generator.  Repetition means
rerunning the same design (including any adaptivity and stopping rule), while holding fixed a finite population or a
schedule of potential outcomes \citep{Fisher1935,Rosenbaum2002,ImbensRubin2015}.  This is the classical logic of
randomization inference and remains central in A/B testing and adaptive experimentation.

Let $(Y_t(0),Y_t(1))_{t\ge1}$ be bounded potential outcomes with $0\le Y_t(0),Y_t(1)\le 1$ a.s.\ for all $t$, and define
the held-fixed $\sigma$-algebra
\[
\mathcal R := \sigma\bigl(\{Y_t(0),Y_t(1):t\ge1\}\bigr).
\]
Let $A_t\in\{0,1\}$ denote treatment assignment and define the observed outcome
\[
X_t := A_tY_t(1)+(1-A_t)Y_t(0),\qquad 0\le X_t\le1.
\]
Let $\mathcal F_t:=\sigma(A_1,X_1,\ldots,A_t,X_t)$ and $\mathcal H_t:=\mathcal R\vee\mathcal F_t$.
Assume the assignment mechanism is sequentially randomized:
\begin{equation}
\label{eq:design-rand-sec6}
\mathbb P(A_t=1\mid \mathcal H_{t-1}) = q_t,\qquad t\ge1,
\end{equation}
where $q_t\in[0,1]$ is $\mathcal F_{t-1}$-measurable (allowing adaptivity).

\begin{proposition}[Design-conditional drift stabilization implies stabilization of observed averages]
\label{prop:design-stabilization}
Under the setup above, define the design-conditional drift $m_t:=\mathbb E[X_t\mid\mathcal H_{t-1}]$.
Then
\[
m_t = q_t\,Y_t(1) + (1-q_t)\,Y_t(0)\qquad\text{a.s.}
\]
If there exists a random variable $\theta$ such that $\frac1n\sum_{t=1}^n m_t \to \theta$ almost surely, then
$\bar X_n\to\theta$ almost surely.
\end{proposition}

\noindent\textit{Proof.} See \hyperref[appendix-A10]{Appendix A.10}. (Apply \autoref{prop:unified-stabilization} with $\mathcal R$ as
defined above.)

\begin{proposition}[Design-based martingale innovations]
\label{prop:design-martingale}
Under the design setup above, define $d_t:=X_t-m_t$ with $m_t=\mathbb E[X_t\mid\mathcal H_{t-1}]$.
Then $(d_t)_{t\ge1}$ is a martingale difference sequence with respect to $(\mathcal H_t)$:
\[
\mathbb E[d_t\mid \mathcal H_{t-1}] = 0 \qquad \text{a.s.\ for all }t\ge1.
\]
If an observable nonnegative supermartingale or e-process is separately shown to be valid under a stated null and the design law, Ville's inequality supplies time-uniform control. Constructing such a process requires additional null, measurability, and moment or conditional-MGF conditions \citep{Ville1939,HowardRamdasMcAuliffeSekhon2020}.
\end{proposition}

\noindent\textit{Proof.} See \hyperref[appendix-A11]{Appendix A.11}. The implications for adaptive A/B testing are illustrated in \autoref{ex:ab}.

\subsection{Fixed-bin calibration from conditional-mean correctness}
\label{sec6.6-fixed-bin-calibration}

A calibration-type guarantee \hyperref[f:F3]{(\ref*{f:F3})} can be obtained from the same drift--innovation logic. If forecasts equal the governing-law conditional means, calibration follows against a \emph{fixed, predictable} checking rule, provided the rule is triggered infinitely often. The point is not to claim calibration as an unconditional frequency property,
but to isolate precisely what this correctness condition delivers and what additional regime structure is still needed
for long-run performance interpretations in adaptive systems.

\medskip

Consider the filtration-based special case where the reference structure is trivial, so
$\mathcal H_t=\mathcal F_t$ and $m_t=\mathbb E[X_t\mid\mathcal F_{t-1}]$.
Assume a probabilistic forecast is conditionally correct under $\mathbb P$:
\[
p_{t-1}:=\mathbb E_{\mathbb P}[X_t\mid\mathcal F_{t-1}]
\qquad\text{(Section~\ref{sec3.3-sequential}).}
\]
Fix a Borel bin $I\subset[0,1]$ and define the predictable checking rule
\[
W_t := \mathbf 1\{p_{t-1}\in I\}\in\{0,1\},
\qquad\text{so that }W_t\text{ is }\mathcal F_{t-1}\text{-measurable.}
\]
Let $N_n(I):=\sum_{t=1}^n W_t$ be the number of visits to bin $I$, and define the binwise empirical frequency and
binwise average forecast (whenever $N_n(I)>0$) by
\[
\widehat{x}_n(I):=\frac{1}{N_n(I)}\sum_{t=1}^n W_t X_t,
\qquad
\widehat{p}_n(I):=\frac{1}{N_n(I)}\sum_{t=1}^n W_t p_{t-1}.
\]

\begin{proposition}[Fixed-bin calibration from conditional-mean correctness, on infinite visitation]
\label{prop:fixed-bin-calibration}
Assume $X_t\in[0,1]$ $\mathbb P$-almost surely and let $p_{t-1}=\mathbb E_{\mathbb P}[X_t\mid\mathcal F_{t-1}]$.
For a fixed Borel bin $I\subset[0,1]$, define $W_t$ and $N_n(I)$ as above. Then the binwise calibration gap admits the ratio form
\begin{equation}
\label{eq:binwise-gap}
\widehat{x}_n(I)-\widehat{p}_n(I)
=
\frac{\sum_{t=1}^n W_t(X_t-p_{t-1})}{\sum_{t=1}^n W_t},
\end{equation}
and, on the event $\{N_n(I)\to\infty\}$,
\begin{equation}
\label{eq:binwise-cal}
\widehat{x}_n(I)-\widehat{p}_n(I)\xrightarrow{\mathrm{a.s.}}0.
\end{equation}
Equivalently, evaluation-law conditional correctness implies calibration on the fixed bin $I$ on the event of infinite visitation.
\end{proposition}

\noindent\textit{Proof.} See Appendix~\ref{appendix-A12}.

\medskip

\autoref{prop:fixed-bin-calibration} is a direct consequence of conditional-mean correctness under the governing law.
Since $W_t$ is predictable and $\mathbb E[X_t-p_{t-1}\mid\mathcal F_{t-1}]=0$, the numerator in
\eqref{eq:binwise-gap} is a martingale transform of the innovation sequence with bounded increments.
Thus, whenever the denominator diverges (i.e., the bin is visited infinitely often), martingale strong-law
arguments force the binwise calibration gap to vanish.

The content of the proposition is therefore intentionally conditional and limited.
First, correctness of the forecasts as governing-law conditional means does \emph{not} ensure the visit condition $N_n(I)\to\infty$.
In adaptive systems, the forecasting rule and decision policy can evolve in response to outcomes,
which can change the distribution of forecasts and hence the set of bins that are visited frequently.
A bin may be visited only finitely often---or visited in bursts separated by long gaps---without violating conditional-mean correctness.
This is the Role Separation Principle in microcosm: evaluation-law conditional correctness supplies martingale cancellation, but it does not supply a regime-level repetition
structure under which calibration can be asserted unconditionally.

Second, \autoref{prop:fixed-bin-calibration} should not be confused with \emph{uniform} calibration over rich
families of bins or data-dependent checking rules. Uniform (or online) calibration typically requires
additional structure or explicit calibration constructions; see, for example, the prequential formulation
in \citet{Dawid1982} and online calibration results in \citet{FosterVohra1998}.

\subsection{Synthesis as a Theorem-Level Map from Reference Structure to Frequency Meaning}
\label{sec6.7-synthesis}

The results above refine the schematic implication \eqref{eq:sep-summary} into an operational template.
The reference--drift decomposition \eqref{eq:drift-innovation-sec6} is universal; the substantive content is the
choice of reference structure and the corresponding drift stabilization claim.

\medskip
\noindent
\begingroup
\setlength{\fboxsep}{6pt}%
\fbox{%
\begin{minipage}{\dimexpr\linewidth-2\fboxsep-2\fboxrule\relax}
\small
\noindent\textbf{Unified stabilization template (reference--drift decomposition).}
Choose (i) a statistic $X_t$, (ii) held-fixed reference information $\mathcal R$, and (iii) a governing law or protocol with the target-specific stability assumption. Define $\mathcal H_t=\mathcal R\vee\mathcal F_t$ and $m_t=\mathbb E_{\mathbb P}[X_t\mid\mathcal H_{t-1}]$. Then

\begin{center}
\begin{adjustbox}{max width=.985\linewidth}
$\textstyle
\underbrace{\frac1n\sum_{t=1}^n (X_t-m_t)\to 0}_
{\substack{\text{\scriptsize innovation cancellation}\\[-0.15em]
\text{\scriptsize (martingale; filtration-level)}}}
\;+\;
\underbrace{\frac1n\sum_{t=1}^n m_t\to\theta}_
{\substack{\text{\scriptsize drift stabilization}\\[-0.15em]
\text{\scriptsize (reference structure; regime-level)}}}
\;\Longrightarrow\;
\underbrace{\bar X_n\to\theta}_{\hyperref[f:F1]{(\ref*{f:F1})}}
$
\end{adjustbox}
\end{center}

If an observable nonnegative supermartingale or e-process is valid under a specified null and reference law, Ville's inequality supplies time-uniform control \citep{Ville1939,HowardRamdasMcAuliffeSekhon2020}. Constructing such a process requires additional null, measurability, and moment or conditional-MGF conditions. A limited fixed-bin calibration guarantee of type \hyperref[f:F3]{(\ref*{f:F3})}, conditional on sufficient visitation, follows from \autoref{prop:fixed-bin-calibration}.
\end{minipage}%
}%
\endgroup
\medskip

Nevertheless, even when stabilization holds under a stated reference structure, identification remains separate
(\autoref{sec5.1-nonid}) because distinct laws can share the same value of a chosen stabilized statistic or the same proper coarsening of its limiting distribution while differing in finite-horizon predictive content or tail risk. Alternative-policy behavior requires a separate causal or structural model.

With this stabilization framework in place, \autoref{sec7-guidance} returns to adaptive and AI-mediated deployments.
There, the operative reference structure is often implicit, partially observed, or itself shaped by the algorithm.
The Principle therefore becomes a reporting discipline in which any claimed stabilization, reliability, or calibration statement must specify the filtration under which probabilities are interpreted and the regime-level reference structure under which the intended frequency meaning is asserted.

\subsection{Relation to Prequential, Adaptive, and Conformal Frameworks}
\label{sec6.8-relation}

Our framework complements prequential and game-theoretic probability, which emphasize sequential scoring rules, betting protocols, and supermartingale tests for pathwise or time-uniform validity \citep{Dawid1982,Dawid1984,Dawid2006,ShaferVovk2001}. Dawid's prequential program is especially close in spirit because it evaluates forecasts through their realized sequential performance, while game-theoretic probability makes the testing protocol explicit by representing evidence as the evolution of a betting capital process. In the language of this paper, these traditions provide powerful filtration-level and pathwise tools; our additional emphasis is on the regime-level question of what repetition structure licenses a frequency interpretation of the resulting forecasts or tests.

Adaptive-data-analysis and post-selection frameworks are related in that they explicitly specify the data-reuse or selection regime under which validity is assessed \citep{DworkEtAl2015,FithianSunTaylor2014,LeeSunSunTaylor2016}. Conformal prediction offers a canonical example of a reference-structure guarantee: standard conformal prediction obtains finite-sample marginal coverage from exchangeability, while work beyond exchangeability modifies the regime through weighting or related devices to accommodate distributional drift \citep{ShaferVovk2008,AngelopoulosBates2021,BarberCandesRamdasTibshirani2023}.

Recent AI reliability studies add complementary motivation. Long-context evaluations show position-dependent use of available information \citep{Liu2024}; multi-turn evaluations show conversational degradation \citep{Laban2025}; and context-drift work uses turn-wise KL divergence between a test model's token-level predictive distribution and that of a goal-consistent reference model \citep{DongreEtAl2025}. These studies document phenomena to which the present framework can be applied diagnostically.

Taken together, these literatures reinforce the central message of this paper: probabilistic meaning in adaptive regimes is inherently regime-conditional. The novelty of the present framework is to make that conditionality the central mathematical object. Instead of adding a new sequential-validity tool or an empirical drift metric, the paper specifies the reference-structure conditions under which filtration-level probability statements can support frequency-level claims. This connects sequential forecasting, adaptive inference, conformal validity, and AI reliability within a common regime-conditional framework.
\section{Interpretive Guidance for Adaptive and AI-mediated Regimes}
\label{sec7-guidance}

With the structural framework established, we turn to its implications for modern adaptive systems. This section translates the Principle into interpretive guidance for applied work in adaptive and AI-mediated settings.
In particular, we treat an ``adaptive environment'' as a protocol in which (i) the environment may adapt to specified
aspects of the past, (ii) the forecaster or decision rule may be deterministic or randomized, and (iii) all validity
statements are formulated relative to an explicit filtration. This protocol-level framing aligns reference structures with
standard notions in online learning, calibration games, and anytime-valid sequential inference.

Our goal is not to propose new algorithms, but to make reliability claims \emph{auditable} by clarifying (i) what
probabilistic object is being invoked, (ii) what repetition mechanism is being assumed, and (iii) which frequency
meaning is actually warranted under the stated structure.

We focus on three recurring sites of confusion: (a) high-confidence errors in generative AI, often described as ``hallucinations'',
(b) calibration claims for probabilistic forecasts, and (c) validity claims for synthetic data. In each case, the
same diagnostic question applies: \emph{what is the filtration-level probability statement, and what reference structure
is being used to justify its intended frequency meaning?}

\subsection{A diagnostic taxonomy: incoherence versus reference-structure failure}
\label{sec7.1-taxonomy}

The Principle can be used as a taxonomy for reliability failures in adaptive systems.  Two failure modes
are conceptually distinct and have different remedies.

\medskip
\noindent\textbf{Mode 1: coherence failures (filtration-level meaning is unclear).}
A coherence failure occurs when a collection of reported assessments cannot be represented by one common probability law on the declared event algebra and information set while respecting logical relations among events. A lone scalar score need not reveal incoherence; it may instead be mismatched to the declared target event or information set. Operational symptoms can include incompatible assessments for related events or time-inconsistent updates (\autoref{sec3-coherence-forecast}).

\medskip
\noindent\textbf{Mode 2: reference-structure failures (frequency meaning is unclear).}
A system can be internally coherent under an assessment law $\mathbb Q$ and yet unreliable under deployment law $\mathbb P$ because the laws differ, the frequency claim lacks a credible repetition mechanism, or both. In the unified stabilization framework, the structural
requirement is drift control: a frequency claim typically presupposes stabilization of the reference drift
$m_t=\mathbb{E}_{\mathbb P}[X_t\mid \mathcal{R}\vee\mathcal{F}_{t-1}]$ or, for time-uniform guarantees, a design/sequential
structure supporting supermartingale control (\autoref{sec6-beyond-exch}).
Mode~2 failures are failures of \emph{regime specification}. A concrete illustration of this regime-specification failure is given in \autoref{ex:mode2}.

In applications, a useful first step is to distinguish these modes.  Addressing a Mode~1 failure typically requires clarifying the probabilistic object and the filtration with respect to which it is conditioned.  Addressing a Mode~2 failure, which is the primary focus here, requires declaring and defending a reference structure that makes the intended frequency interpretation meaningful.

\subsection{Hallucinations as event-space mismatch and regime drift}
\label{sec7.2-hallucinations}

The term ``hallucination'' is used broadly to describe fluent but unsupported outputs in neural text generation and LLMs \citep{Ji2023}.  In this paper, hallucination is not treated as a psychological metaphor, but as
a mismatch between a probabilistic object and the frequency interpretation implicitly demanded of it.

Consider an interactive system used sequentially.  At time $t$, the system has access to a history filtration
$\mathcal{F}_t$ that may include prompts, retrieved documents, tool outputs, and its own past actions.  It produces an
output $O_{t+1}$ and possibly a scalar ``confidence'' $p_t\in[0,1]$ that users interpret as a probability of a
correctness event
\[
X_{t+1} := \mathbf{1}\{\text{$O_{t+1}$ is correct under a stated evaluation rule}\}.
\]
A coherence interpretation would require an assessment law $\mathbb Q$ under which $p_t=\mathbb E_{\mathbb Q}[X_{t+1}\mid\mathcal F_t]$, or a stated decision problem under which the score is Bayes-optimal. By contrast, statements such as ``the model is
calibrated'' or ``confidence reflects correctness frequency'' are frequency statements of type
\hyperref[f:F3]{(\ref*{f:F3})} and require a declared reference structure defining repetition and stability.

The Principle highlights two structural mismatches relevant to some high-confidence unsupported outputs. It does not claim that they exhaust the empirical causes of hallucinations.

\medskip
\noindent\textbf{(H1) Event-space mismatch: a coherent text model is not automatically a correctness model.}
Many generative systems are trained to model text. An autoregressive distribution over strings does not by itself identify or supply a conditional probability of correctness. Such a probability requires a jointly specified model containing external truth or evaluation variables and the mapping from output to correctness. A coherent probability can be defined on such an extended space, but high string likelihood alone is not high truth probability. \autoref{ex:halluc-cite} gives an illustration.

\medskip
\noindent\textbf{(H2) Regime drift: even if correctness is modeled, frequency meaning depends on stability.}
Suppose first that a system outputs the exact evaluation-law conditional mean $p_t=\mathbb E_{\mathbb P}[X_{t+1}\mid\mathcal F_t]$. Then \eqref{eq:vitl} follows under $\mathbb P$. For an approximate score $\widetilde p_t$, the average approximation error $n^{-1}\sum_{t<n}(p_t-\widetilde p_t)$ must also be controlled. Even under exact conditional correctness, long-run correctness
frequencies still depend on the regime-level drift term in the decomposition
\eqref{eq:drift-innovation-sec6}.  In interactive deployments, prompt distributions evolve, users adapt to prior model
outputs, and the model itself may be updated.  These are precisely the mechanisms that can prevent drift stabilization.
In such regimes, ``calibration'' cannot be treated as a model-internal property; it is a joint property of the coupled
system generating $(p_t,X_{t+1})$ under a stated reference structure. This point aligns with evidence that long-context and multi-turn models can use available context unreliably and that a test model can exhibit turn-wise KL divergence from a goal-consistent reference model \citep{Liu2024,Laban2025,DongreEtAl2025}.

The practical implication is that reducing high-confidence errors is not simply a matter of ``making probabilities coherent.''
It requires aligning the probabilistic object with the event and filtration users intend (addressing (H1)), and it
requires evaluating the system under an explicit and credible reference structure (addressing (H2)).

\subsection{Calibration claims: what can be asserted, and under what assumptions}
\label{sec7.3-calibration}

Calibration is often treated as a canonical reliability criterion for probabilistic forecasts
\citep{Dawid1982,GneitingRaftery2007}.  In the language of \autoref{sec2.3-frequency}, calibration is a frequency claim
of type \hyperref[f:F3]{(\ref*{f:F3})}: predicted probabilities should match empirical frequencies within a specified
reference class of checking rules.  In adaptive deployments, two over-interpretations are especially common, and
\autoref{prop:fixed-bin-calibration} provides a useful anchor for what evaluation-law conditional correctness can and cannot justify.

\medskip
\noindent\textbf{(C1) Validity in the large does not imply unconditional or rich-class calibration.}
Evaluation-law conditional correctness yields \eqref{eq:vitl}, controlling a single global discrepancy along the realized
trajectory.  This is generally insufficient for binwise or conditional calibration, because calibration evaluates
discrepancies on data-dependent subsets of time (\autoref{rem:vitl-vs-calibration}).
For a fixed predictable check there is a limited statement. By \autoref{prop:fixed-bin-calibration}, for any fixed Borel bin $I\subset[0,1]$ and $W_t=\mathbf 1\{p_{t-1}\in I\}$, the binwise calibration gap satisfies
$\widehat{x}_n(I)-\widehat{p}_n(I)\xrightarrow{\mathrm{a.s.}}0$ on $\{N_n(I)\to\infty\}$. This is an \hyperref[f:F3]{(\ref*{f:F3})}-type conclusion, but it is deliberately conditional and non-uniform because it concerns a pre-declared check and requires that the corresponding denominator diverge. The distinction between global innovation cancellation and binwise calibration failure is illustrated in \autoref{ex:global-not-cal}.

\medskip
\noindent\textbf{(C2) ``The model is calibrated'' is incomplete without a checking class and a regime.}
Even in static settings, calibration depends on what is checked and how forecasts are binned.  In adaptive settings, the
additional subtlety is that the visit behavior of the checks is regime-dependent. In particular, conditional-mean correctness
does not ensure the condition $N_n(I)\to\infty$ in \autoref{prop:fixed-bin-calibration}. Feedback, policy updates, user
adaptation, or model updates can shift the distribution of forecasts and render some bins rarely (or only finitely)
visited.  Thus, calibration cannot be treated as a purely model-internal attribute; it is a joint property of the coupled
process generating $(p_t,X_{t+1})$ under a declared reference structure.

A useful reporting discipline is therefore to state calibration claims in the form
\[
\text{``$(p_t,X_{t+1})$ is calibrated with respect to checking class $\mathcal C$ under reference structure $\mathfrak R$.''}
\]
Here, $\mathcal C$ may encode binning rules, group membership, context restrictions, abstention/selection rules, or other
evaluation protocols, and $\mathfrak R$ specifies what is held fixed under ``repetition'' (e.g., a frozen model window, a
fixed logging policy, a stationary environment, or a design-based audit cohort).  The practical implications of this
regime-conditional calibration discipline are illustrated in \autoref{ex:selective}.

\medskip

Finally, \autoref{prop:filtration-stabilization} provides a precise bridge to frequency meaning of type
\hyperref[f:F1]{(\ref*{f:F1})}. If evaluation-law conditionally correct one-step means stabilize ($p_t\to p_\infty$ a.s.), then the empirical average
stabilizes ($\bar X_n\to p_\infty$ a.s.).  This supports a coarse sense in which long-run outcomes match stabilized
beliefs.  However, stabilization of $(p_t)$ does not by itself imply calibration
\hyperref[f:F3]{(\ref*{f:F3})} over nontrivial checking classes.  Outside the fixed-check setting of
\autoref{prop:fixed-bin-calibration}, binwise or conditional calibration typically requires additional assumptions that
restrict the checking protocol or control adaptivity, as in prequential calibration theory \citep{Dawid1982} and online
calibration constructions \citep{FosterVohra1998}.  The key message is that calibration claims must always specify (i) the
checking class and (ii) the reference regime under which the denominators and frequencies are intended to be stable.

\subsection{Synthetic data: coherence of the generator is not validity of inference}
\label{sec7.4-synth}

Synthetic data generation sharpens the Principle by making the probabilistic generator explicit.
Let $\mathbb{P}$ denote the (typically unknown) law of the real data, and let $\mathbb{Q}$ denote the law induced by a
synthetic generator.  Drawing synthetic datasets corresponds to sampling from $\mathbb{Q}$; repeated-sampling
properties of analyses on synthetic data are therefore frequency statements under $\mathbb{Q}$, not automatically under
$\mathbb{P}$ \citep{Rubin1993,Reiter2005,Drechsler2011}.

Two separations explain why synthetic data can fail structurally despite being produced by a coherent generator.

\medskip
\noindent\textbf{(S1) Identification separation: matching summaries does not identify the law.}
As shown in \autoref{sec5.1-nonid}, coarse frequency behavior does not uniquely determine the probability law.
In practice, synthetic data is often validated by matching a finite collection of marginals or low-order summaries.
Passing these checks is not sufficient to guarantee preservation of conditional structure needed for downstream
inference (effects, tail behavior, dependence relevant for uncertainty quantification).  Frequency matching is a
projection, not a definition, of the underlying law. A concrete case showing how matched summaries can conceal distortions in conditional structure appears in \autoref{ex:synth-cond}.

\medskip
\noindent\textbf{(S2) Reference-structure separation: validity is a cross-regime claim.}
Statements such as ``inference on synthetic data has correct coverage for the real-world estimand'' are frequency
claims of type \hyperref[f:F2]{(\ref*{f:F2})} that require a declared replication
scheme connecting $\mathbb{Q}$ to $\mathbb{P}$.  One must state what is held fixed (training data, privacy mechanism,
synthesizer architecture, analyst behavior) and what is repeated (synthetic draws, retraining, adaptive release
schedule).  In adaptive pipelines, the generator itself may drift (e.g., periodic retraining), and synthetic data may
enter feedback loops that alter future data collection, undermining the stability required for frequency interpretation. An example highlighting how generator updates and feedback loops compromise cross-regime validity appears in \autoref{ex:synth-drift}.

The methodological implication is that synthetic data does not eliminate the need for reference structures; it makes
them unavoidable.  Validity claims should be stated as regime-conditional statements that (i) define the target
estimand and evaluation protocol, (ii) specify the generator and training/retraining mechanism, and (iii) articulate
the stability route under which the intended frequency meaning is justified.

\subsection{Practical criteria for articulating and reporting probabilistic guarantees}
\label{sec7.5-criteria}

We conclude with a compact set of criteria that operationalize the Principle as a structural discipline for interpreting and evaluating probabilistic claims in adaptive regimes. These criteria render such claims explicitly regime-conditional and therefore open to coherent scrutiny and empirical audit. For auditability, any such claim should state (i) the target event/statistic and evaluation rule, (ii) the filtration or effective information set, (iii) the probabilistic object and law, (iv) the intended frequency meaning (\hyperref[f:F1]{(\ref*{f:F1})}--\hyperref[f:F3]{(\ref*{f:F3})}) and checking class when relevant, (v) the reference structure defining repetition, (vi) the theorem-level stability route, and (vii) the conditions whose violation would invalidate the claim. The same seven items appear in the concise checklist below and in expanded form in \hyperref[app:reporting-checklist]{Appendix~D.2}.
\begin{center}
\fbox{\parbox{0.9\linewidth}{
\textbf{Regime-Conditional Audit Checklist}
\begin{enumerate}
\item What target event, statistic, and evaluation rule are declared?
\item What filtration or effective information set is used?
\item What probabilistic object is reported, and under which law?
\item What frequency meaning and, for calibration, what checking class are claimed?
\item What reference structure defines repetition?
\item What invariance or controlled-drift assumptions and theorem-level mechanism justify the interpretation?
\item What violations would invalidate the claim?
\end{enumerate}
}}
\end{center}
\hyperref[appendix-E]{Appendix E} provides two recommender-system illustrations covering policy-value stabilization and fixed-bin calibration under a declared audit regime.

\section{Conclusion}
\label{sec8-conclusion}

Modern adaptive systems do not weaken probability; they expose its structure. In textbook i.i.d.\ settings, internal assessment, evaluation-law correctness, and several frequency targets are often studied under one stable regime, although they remain logically distinct. Adaptive feedback makes their separation unavoidable. Probability serves as a filtration-relative assessment language and, with a stated loss or utility, supports decisions; empirical claims such as stabilization, calibration, and repeated-sampling validity require further law- and regime-specific premises.

This message connects prequential forecasting and calibration \citep{Dawid1982,Dawid1984,Dawid2006,FosterVohra1998}, game-theoretic and martingale validity \citep{ShaferVovk2001,HowardRamdasMcAuliffeSekhon2020}, adaptive data analysis \citep{DworkEtAl2015}, conformal prediction \citep{ShaferVovk2008,AngelopoulosBates2021,BarberCandesRamdasTibshirani2023}, and empirical AI-reliability studies \citep{Ji2023,Maynez2020,Liu2024,Laban2025,DongreEtAl2025}. The contribution here is a unified account of the distinct law, filtration, repetition, and stability premises these uses require.

Internal coherence under $\mathbb Q$ gives probability its filtration-relative assessment meaning. Martingale innovation cancellation under $\mathbb P$ additionally requires conditional-mean correctness under that governing law. This condition yields validity in the large and fixed-bin calibration on infinite visitation, but it does not by itself give unconditional stabilization, nonvacuous or rich-class calibration, or procedure-level error control. Those claims require target-specific reference and sufficient stability conditions. Within the binary exchangeable class, the full distribution of $L$ identifies the exchangeable law; the non-identification result concerns a proper coarsening of that distribution. Conditional-mean correctness under a governing law can still coexist with almost-sure nonstabilization.

Our constructive contribution is a unified stabilization framework that decomposes empirical averages into a predictable reference drift and a martingale innovation. Stabilization follows when the drift average converges under the declared regime. Time-uniform control is separate and requires an observable nonnegative supermartingale or e-process valid under a specified null and reference law.

For AI-mediated systems, the framework supplies diagnostic hypotheses rather than established causes or remedies. Grounding, memory summaries, re-anchoring, calibration audits, and policy--environment controls are candidate interventions whose efficacy must be established empirically in the relevant regime.

Reliability claims are not properties of a probability number alone. A defensible statement specifies the target, assessment and evaluation laws when distinct, governing filtration, intended frequency interpretation, repetition protocol, checking class where relevant, and stability mechanism. Making these commitments explicit clarifies the scope of each guarantee and makes its empirical and theoretical support assessable.

\selectlanguage{english}

\begin{singlespace}
\newpage
\clearpage 
{
\setlength{\parskip}{1em} 
\setlength{\parindent}{0pt} 

\appendix
\refstepcounter{section}
\section*{Appendix A. Proofs of Propositions and an Auxiliary Theorem}
\phantomsection
\addcontentsline{toc}{section}{Appendix A. Proofs of Propositions and an Auxiliary Theorem}
\label{appendix-A}

This appendix collects proofs of the propositions stated in the main text (Sections~\ref{sec3-coherence-forecast}--\ref{sec6-beyond-exch}) together with the auxiliary de~Finetti representation (\autoref{them:definetti}) referenced in \autoref{sec4.2-definetti}. Several results are classical (de~Finetti, Birkhoff's ergodic theorem, Markov chain ergodic theorems); for these we provide short arguments and precise citations.

\subsection{Proof of \autoref{prop:coherence-decision}}
\label{appendix-A1}

\noindent\emph{(i) No nonzero one-sided payoff at a fixed information set.}
Let $Y_1,\ldots,Y_m\in L^1(\mathbb Q)$ and let $H_1,\ldots,H_m$ be bounded and $\mathcal G$-measurable.
Define
\[
S:=\sum_{i=1}^m H_i\bigl(Y_i-\mathbb E_{\mathbb Q}[Y_i\mid\mathcal G]\bigr).
\]
Then $S\in L^1(\mathbb Q)$. By linearity and the pull-out property,
\[
\mathbb E_{\mathbb Q}[S\mid\mathcal G]
=
\sum_{i=1}^m H_i\Bigl(\mathbb E_{\mathbb Q}[Y_i\mid\mathcal G]-\mathbb E_{\mathbb Q}[Y_i\mid\mathcal G]\Bigr)
=0
\qquad\text{a.s.}
\]
Taking expectations yields $\mathbb E_{\mathbb Q}[S]=0$. If $S\ge0$ $\mathbb Q$-a.s. with $\mathbb Q(S>0)>0$, integrability would imply $\mathbb E_{\mathbb Q}[S]>0$; the opposite one-sided sign would imply a negative expectation. Both contradict zero expectation.

\medskip
\noindent\emph{(ii) Conditional probability is Bayes-optimal under Brier loss.}
Fix $A\in\mathcal{F}$, write $Y:=\mathbf{1}_A$, and let $q\in L^0(\mathcal{G};[0,1])$ be any $\mathcal{G}$-measurable forecast.
Let $p:=\mathbb E_{\mathbb Q}[Y\mid\mathcal G]=\mathbb Q(A\mid\mathcal G)$.
Since $0\le Y,q\le1$, the loss $(Y-q)^2$ is integrable.  Expand
\[
(Y-q)^2
=
(Y-p+p-q)^2
=
(Y-p)^2 + (q-p)^2 + 2(p-q)(Y-p).
\]
Taking $\mathbb Q$-conditional expectations given $\mathcal G$ and using $\mathbb E_{\mathbb Q}[Y-p\mid\mathcal G]=0$ yields
\[
\mathbb E_{\mathbb Q}[(Y-q)^2\mid\mathcal G]
=
\mathbb E_{\mathbb Q}[(Y-p)^2\mid\mathcal G] + (q-p)^2
\qquad\text{a.s.}
\]
Thus the conditional expected Brier loss is minimized uniquely up to $\mathbb Q$-null sets at $q=p$.
\qed

\subsection{Proof of \autoref{prop:exch-align}}
\label{appendix-A2}

Let $(X_t)_{t\ge1}$ be exchangeable and let $\mathcal{F}_t:=\sigma(X_1,\ldots,X_t)$.
By \autoref{them:definetti}, there exists $\Theta\in[0,1]$ such that conditional on $\Theta$ the sequence is i.i.d.\ $\mathrm{Bernoulli}(\Theta)$.

\medskip
\noindent\emph{(i) Predictive martingale and convergence.}
Define $p_t:=\mathbb{P}(X_{t+1}=1\mid\mathcal{F}_t)=\mathbb{E}[X_{t+1}\mid\mathcal{F}_t]$.
By iterated expectation and conditional i.i.d.\ given $\Theta$,
\[
p_t
=
\mathbb{E}\!\left[\mathbb{E}[X_{t+1}\mid \Theta,\mathcal{F}_t]\mid \mathcal{F}_t\right]
=
\mathbb{E}[\Theta\mid \mathcal{F}_t].
\]
Hence $(p_t)_{t\ge0}$ is a bounded martingale with respect to $(\mathcal{F}_t)$.
By the martingale convergence theorem (or L\'evy's upward theorem), $p_t\to \mathbb{E}[\Theta\mid \mathcal{F}_\infty]$ a.s.,
where $\mathcal{F}_\infty:=\sigma(X_1,X_2,\ldots)$.
In the de~Finetti representation one may take $\Theta$ to be measurable with respect to the exchangeable (hence $\mathcal{F}_\infty$-measurable) $\sigma$-field, so $\mathbb{E}[\Theta\mid\mathcal{F}_\infty]=\Theta$ a.s., and therefore $p_t\to\Theta$ a.s.

\medskip
\noindent\emph{(ii) Stabilization of empirical averages.}\par\noindent
Conditional on $\Theta=\theta$, the variables $(X_t)$ are i.i.d.\ Bernoulli with parameter $\theta$, so by the strong law of large numbers,
$\bar X_n:=\frac1n\sum_{t=1}^n X_t \to \theta$ almost surely under $\mathbb{P}(\cdot\mid\Theta=\theta)$.
Equivalently, $\bar X_n\to\Theta$ almost surely under $\mathbb{P}$.
\qed

\subsection{Proof of \autoref{prop:freq-not-law}}
\label{appendix-A3}

\noindent\textbf{Part (i).}
Let $\mathbb{P}$ be exchangeable on $(\mathcal{X},\mathcal{B}(\mathcal{X}))$.
By de Finetti’s representation (\autoref{them:definetti}), there exists $\Theta\in[0,1]$ with law $\mu$ such that conditional on $\Theta$ the coordinates are i.i.d.\ $\mathrm{Bernoulli}(\Theta)$.
By the strong law, $\bar X_n\to\Theta$ almost surely, hence the globally defined limiting-frequency functional $L$ satisfies $L=\Theta$ almost surely.
Therefore, for any $p\in[0,1]$,
\[
\mathbb{P}(L=p)=\mathbb{P}(\Theta=p)=\mu(\{p\}),
\]
which is \eqref{eq:atom-mass}.

\medskip
\noindent\textbf{Part (ii).}
Fix $p\in(0,1)$ and $\alpha\in(0,1)$.
Choose two distinct points $p_1,p_2\in(0,1)$ with $p_1\neq p_2$ and $p_i\neq p$, and define two distinct mixing measures
\[
\mu_1:=\alpha\,\delta_p+(1-\alpha)\,\delta_{p_1},
\qquad
\mu_2:=\alpha\,\delta_p+(1-\alpha)\,\delta_{p_2}.
\]
Let $\mathbb{P}_1,\mathbb{P}_2$ be the corresponding exchangeable laws via the mixture representation in \autoref{them:definetti}.
Then $\mathbb{P}_1\neq\mathbb{P}_2$ because the de~Finetti mixing measure is uniquely determined by the exchangeable law.
However, in both cases $L=\Theta$ a.s., so $\mathbb{P}_i(L=p)=\mu_i(\{p\})=\alpha$ for $i=1,2$.
\qed

\subsection{Proof of \autoref{prop:coh-not-stable}}
\label{appendix-A4}

Let $\mathcal F_t=\sigma(X_1,\ldots,X_t)$ and define $p_t:=\mathbb E_{\mathbb R}[X_{t+1}\mid\mathcal F_t]$. Then $(X_{t+1}-p_t)_{t\ge0}$ is a bounded martingale-difference sequence under $\mathbb R$; hence the bounded martingale strong law gives the validity statement in the proposition. It remains to construct a non-degenerate, countably additive law $\mathbb R$ on $\mathcal X=\{0,1\}^{\mathbb N}$ for which $\mathbb R(D)=0$.

\medskip
\noindent\emph{Step 1: a block-constant process.}
Let $L_k:=2^{k^2}$ and $T_k:=\sum_{j=1}^k L_j$.
Let $(\Xi_k)_{k\ge1}$ be i.i.d.\ $\mathrm{Bernoulli}(1/2)$ and define a sequence $(X_t)_{t\ge1}$ by
\[
X_t := \Xi_k
\qquad\text{for }T_{k-1}< t\le T_k,\quad k\ge1,
\]
with $T_0:=0$.  Let $\mathbb{R}$ denote the law of $(X_t)$ on $(\mathcal{X},\mathcal{B}(\mathcal{X}))$.
Then $\mathbb{R}$ is countably additive and non-degenerate (e.g.\ $\mathbb{R}(X_1=1)=1/2$).

\medskip
\noindent\emph{Step 2: the empirical mean tracks the last block.}
Write $\bar X_{T_k}=\frac{1}{T_k}\sum_{t=1}^{T_k} X_t$.
Since the $k$th block is constant with value $\Xi_k$,
\[
\bar X_{T_k}
=
\frac{T_{k-1}}{T_k}\,\bar X_{T_{k-1}}
+
\frac{L_k}{T_k}\,\Xi_k,
\]
so
\[
\bigl|\bar X_{T_k}-\Xi_k\bigr|
\le
\frac{T_{k-1}}{T_k}.
\]
Moreover,
\[
\frac{T_{k-1}}{T_k}
=
\frac{\sum_{j<k}2^{j^2}}{\sum_{j\le k}2^{j^2}}
\le
\frac{k\,2^{(k-1)^2}}{2^{k^2}}
=
k\,2^{-2k+1}\xrightarrow[k\to\infty]{}0.
\]
Hence $\bar X_{T_k}-\Xi_k\to0$ almost surely.

\medskip
\noindent\emph{Step 3: infinitely many zeros and ones force non-convergence.}
Since $(\Xi_k)$ are i.i.d.\ $\mathrm{Bernoulli}(1/2)$, with probability one the sequence takes the values $0$ and $1$ infinitely often.
Along the subsequence $(T_k)$, the empirical means $\bar X_{T_k}$ therefore have subsequential limits $0$ and $1$.
Consequently $\bar X_n$ cannot converge, so $\mathbb{R}(D)=0$.
\qed

\subsection{Proof of \autoref{prop:unified-stabilization}}
\label{appendix-A5}

Let $\mathcal{H}_t:=\mathcal{R}\vee\mathcal{F}_t$ and $m_t:=\mathbb{E}[X_t\mid\mathcal{H}_{t-1}]$.
Define $d_t:=X_t-m_t$ and $M_n:=\sum_{t=1}^n d_t$.
Then $(M_n)_{n\ge0}$ is a martingale with respect to $(\mathcal{H}_t)$ and, since $0\le X_t,m_t\le1$, we have $|d_t|\le1$ a.s.

Fix $\varepsilon>0$.  By Azuma--Hoeffding for martingales with bounded increments,
\[
\mathbb{P}\bigl(|M_n|\ge n\varepsilon\bigr)
\le
2\exp\!\left(-\frac{\varepsilon^2}{2}n\right).
\]
The series $\sum_{n\ge1}\mathbb{P}(|M_n|\ge n\varepsilon)$ converges, so by Borel--Cantelli,
$|M_n|/n<\varepsilon$ for all but finitely many $n$ almost surely.
Since $\varepsilon$ is arbitrary, $M_n/n\to0$ a.s.
Thus
\[
\bar X_n-\frac1n\sum_{t=1}^n m_t
=
\frac1n\sum_{t=1}^n (X_t-m_t)
=
\frac{M_n}{n}\xrightarrow{\mathrm{a.s.}}0.
\]
Combining with the drift assumption \eqref{eq:cesaro-drift-sec6} yields $\bar X_n\to\theta$ almost surely.
\qed

\subsection{Proof of \autoref{prop:filtration-stabilization}}
\label{appendix-A6}

Apply \autoref{prop:unified-stabilization} with the trivial reference $\sigma$-algebra $\mathcal{R}$, so that $\mathcal{H}_t=\mathcal{F}_t$.
Then $m_t=\mathbb{E}[X_t\mid\mathcal{F}_{t-1}]=p_{t-1}$ for $t\ge1$.
If $p_t\to p_\infty$ a.s., then by Ces\`aro's theorem,
\[
\frac1n\sum_{t=1}^n m_t
=
\frac1n\sum_{t=1}^n p_{t-1}
\xrightarrow{\mathrm{a.s.}} p_\infty.
\]
\autoref{prop:unified-stabilization} then yields $\bar X_n\to p_\infty$ almost surely.
\qed

\subsection{Proof of \autoref{prop:ville-inequality}}
\label{appendix-A7}

Fix $c>0$ and define the (possibly infinite) stopping time $\tau:=\inf\{t\ge0:\,M_t\ge c\}$ and the bounded stopping times $\tau_n:=\min(\tau,n)$.
Since $(M_t)$ is a nonnegative supermartingale with $M_0=1$, optional stopping for bounded stopping times gives
$\mathbb{E}[M_{\tau_n}]\le \mathbb{E}[M_0]=1$.
On $\{\tau\le n\}$ we have $M_{\tau_n}=M_\tau\ge c$, while on $\{\tau>n\}$ we have $M_{\tau_n}=M_n\ge0$.
Therefore
\[
1\ge \mathbb{E}[M_{\tau_n}]
\ge \mathbb{E}\!\left[M_{\tau_n}\mathbf{1}_{\{\tau\le n\}}\right]
\ge c\,\mathbb{P}(\tau\le n).
\]
Hence $\mathbb{P}(\tau\le n)\le 1/c$ for all $n$, and letting $n\to\infty$ yields
$\mathbb{P}(\sup_{t\ge0}M_t\ge c)=\mathbb{P}(\tau<\infty)\le 1/c$.
\qed

\subsection{Proof of \autoref{prop:shift-stabilization}}
\label{appendix-A8}

If $(X_t)$ is strictly stationary with $\mathbb{E}|X_1|<\infty$, Birkhoff's pointwise ergodic theorem yields
\[
\bar X_n=\frac1n\sum_{t=1}^n X_t \xrightarrow{\mathrm{a.s.}} \mathbb{E}[X_1\mid\mathcal{I}],
\]
where $\mathcal{I}$ is the shift-invariant $\sigma$-algebra.
If the process is ergodic, $\mathcal{I}$ is trivial and thus $\mathbb{E}[X_1\mid\mathcal{I}]=\mathbb{E}[X_1]$ a.s.
\qed

\subsection{Proof of \autoref{prop:policy-stabilization}}
\label{appendix-A9}

Let $\mathcal{F}_t$ be the natural filtration generated by the trajectory up to time $t$.
By the Markov property and the definition of $r$ in \autoref{prop:policy-stabilization},
\[
\mathbb{E}[X_{t+1}\mid \mathcal{F}_t]=\mathbb{E}[X_{t+1}\mid S_t,A_t]=r(S_t,A_t)\qquad\text{a.s.}
\]
By assumption, $(S_t,A_t)$ is a positive Harris recurrent, aperiodic Markov chain with invariant probability measure $\nu_{\pi,\mathsf P}$.
The Markov chain ergodic theorem implies
\[
\frac1n\sum_{t=0}^{n-1} r(S_t,A_t)
\longrightarrow
\int_{\mathsf S\times\mathsf A} r(s,a)\,\nu_{\pi,\mathsf P}(ds,da)
=:\rho(\pi,\mathsf P)
\qquad\text{a.s.}
\]
Since $0\le X_t\le1$, \autoref{prop:unified-stabilization} (with trivial $\mathcal R$) then yields
$\bar X_n\to \rho(\pi,\mathsf P)$ almost surely.
\qed

\subsection{Proof of \autoref{prop:design-stabilization}}
\label{appendix-A10}

By definition and the fact that $Y_t(0),Y_t(1)$ are $\mathcal{R}$-measurable,
\[
m_t
=\mathbb{E}[X_t\mid\mathcal{H}_{t-1}]
=
\mathbb{E}\!\left[A_tY_t(1)+(1-A_t)Y_t(0)\mid\mathcal{H}_{t-1}\right]
=
Y_t(1)\,\mathbb{E}[A_t\mid\mathcal{H}_{t-1}]
+
Y_t(0)\,\mathbb{E}[1-A_t\mid\mathcal{H}_{t-1}].
\]
Using $\mathbb{E}[A_t\mid\mathcal{H}_{t-1}]=\mathbb{P}(A_t=1\mid\mathcal{H}_{t-1})=q_t$ (by \eqref{eq:design-rand-sec6}) gives
$m_t=q_tY_t(1)+(1-q_t)Y_t(0)$ a.s.
If $(1/n)\sum_{t=1}^n m_t\to\theta$ a.s., \autoref{prop:unified-stabilization} applied with the reference
$\sigma$-algebra $\mathcal{R}$ and augmented filtration $\mathcal{H}_t=\mathcal{R}\vee\mathcal{F}_t$ yields $\bar X_n\to\theta$ a.s.
\qed

\subsection{Proof of \autoref{prop:design-martingale}}
\label{appendix-A11}

By definition, $m_t=\mathbb{E}[X_t\mid\mathcal{H}_{t-1}]$ is $\mathcal{H}_{t-1}$-measurable and
\[
\mathbb{E}[X_t-m_t\mid\mathcal{H}_{t-1}]
=
\mathbb{E}[X_t\mid\mathcal{H}_{t-1}]-m_t
=
0\qquad\text{a.s.}
\]
Thus $d_t:=X_t-m_t$ is a martingale difference sequence with respect to $(\mathcal{H}_t)$.
If $(M_t)_{t\ge0}$ is any nonnegative supermartingale with respect to $(\mathcal{H}_t)$ with $M_0=1$, then by \autoref{prop:ville-inequality},
\[
\mathbb{P}\!\left(\sup_{t\ge0}M_t\ge c\right)\le \frac1c,\qquad c>0,
\]
and for any possibly infinite stopping time $T$,
\[
\mathbb P(T<\infty,M_T\ge c)
\le\mathbb P\!\left(\sup_{t\ge0}M_t\ge c\right)
\le\frac1c.
\]
If $T<\infty$ almost surely, the first event may be written $\{M_T\ge c\}$. This implication applies only after a suitable observable null-valid supermartingale has been constructed.
\qed

\subsection{Proof of \autoref{prop:fixed-bin-calibration}}
\label{appendix-A12}

Let $(\mathcal F_t)_{t\ge0}$ be a filtration and assume $X_t\in[0,1]$ a.s. for all $t\ge1$.
Define the conditionally correct forecast process $p_{t-1}:=\mathbb E_{\mathbb P}[X_t\mid\mathcal F_{t-1}]$ and the innovation sequence
\[
Z_t := X_t-p_{t-1}.
\]
Then $\mathbb E[Z_t\mid\mathcal F_{t-1}]=0$ a.s. for each $t$ and $|Z_t|\le 1$ a.s. (since $0\le X_t,p_{t-1}\le 1$).

Fix a Borel bin $I\subset[0,1]$ and let $W_t:=\mathbf 1\{p_{t-1}\in I\}$, which is $\mathcal F_{t-1}$-measurable.
Define
\[
N_n(I):=\sum_{t=1}^n W_t,
\qquad
M_n := \sum_{t=1}^n W_t Z_t = \sum_{t=1}^n W_t(X_t-p_{t-1}).
\]
Then $(M_n)_{n\ge1}$ is an $(\mathcal F_n)$-martingale, because
\[
\mathbb E[M_n-M_{n-1}\mid\mathcal F_{n-1}]
=
\mathbb E[W_n Z_n\mid\mathcal F_{n-1}]
=
W_n\,\mathbb E[Z_n\mid\mathcal F_{n-1}]
=
0
\qquad\text{a.s.}
\]
Moreover, the increments are bounded: $|M_n-M_{n-1}|=|W_n Z_n|\le 1$ a.s.

\medskip
\noindent\emph{Step 1: Ratio representation.}
For $n$ with $N_n(I)>0$, the definitions of $\widehat x_n(I)$ and $\widehat p_n(I)$ give
\[
\widehat{x}_n(I)-\widehat{p}_n(I)
=
\frac{\sum_{t=1}^n W_t X_t}{N_n(I)}-\frac{\sum_{t=1}^n W_t p_{t-1}}{N_n(I)}
=
\frac{\sum_{t=1}^n W_t (X_t-p_{t-1})}{N_n(I)}
=
\frac{M_n}{N_n(I)},
\]
which is \eqref{eq:binwise-gap}.

\medskip
\noindent\emph{Step 2: A predictably weighted martingale.}
Define
\[
a_t:=\frac{W_t}{N_{t-1}(I)+1},
\qquad
L_n:=\sum_{t=1}^n a_tZ_t.
\]
Because $W_t$ and $N_{t-1}(I)$ are $\mathcal F_{t-1}$-measurable, $(a_t)$ is predictable and $(L_n)$ is a martingale. At the $k$th visit to $I$, the nonzero coefficient is exactly $1/k$. Hence
\[
\mathbb E_{\mathbb P}[L_n^2]
=\sum_{t=1}^n\mathbb E_{\mathbb P}[a_t^2Z_t^2]
\le\mathbb E_{\mathbb P}\!\left[\sum_{t=1}^n a_t^2\right]
\le\sum_{k=1}^{\infty}\frac1{k^2}<\infty.
\]
Thus $(L_n)$ is uniformly bounded in $L^2$ and converges almost surely.

On the event $\{N_n(I)\to\infty\}$, let $\tau_k$ denote the finite time of the $k$th visit. Along this event,
\[
\sum_{t=1}^{\infty}a_tZ_t
=\sum_{k=1}^{\infty}\frac{Z_{\tau_k}}{k}
\]
converges. Kronecker's lemma therefore gives
\[
\frac1k\sum_{j=1}^k Z_{\tau_j}\longrightarrow0.
\]
If $N_n(I)=k$, then $\sum_{t=1}^nW_tZ_t=\sum_{j=1}^kZ_{\tau_j}$. Since $k=N_n(I)\to\infty$ on the event under consideration,
\[
\frac{\sum_{t=1}^nW_t(X_t-p_{t-1})}{N_n(I)}\longrightarrow0.
\]
Together with the ratio identity, this proves \eqref{eq:binwise-cal}. \qed

\subsection{Auxiliary result: de~Finetti representation (\autoref{sec4.2-definetti})}
\label{appendix-A13}

\begin{theorem}[de~Finetti representation; binary case]
\label{them:definetti}
If $(X_t)_{t\ge1}$ is an exchangeable $\{0,1\}$-valued sequence, then there exists a random variable
$\Theta\in[0,1]$ such that, conditional on $\Theta$, the variables $(X_t)_{t\ge1}$ are i.i.d.\ $\mathrm{Bernoulli}(\Theta)$.
Equivalently, there exists a Borel probability measure $\mu$ on $[0,1]$ such that for every $n\ge1$ and every
$(x_1,\ldots,x_n)\in\{0,1\}^n$,
\[
\mathbb{P}(X_1=x_1,\ldots,X_n=x_n)
=
\int_0^1 \prod_{t=1}^n \theta^{x_t}(1-\theta)^{1-x_t}\,\mu(d\theta).
\]
The measure $\mu$ is uniquely determined by the exchangeable law and is the distribution of $\Theta$.
\end{theorem}

\noindent\textit{Proof.} Classical; see \citet{kallenberg2005probabilistic}. \qed

\refstepcounter{section}
\section*{Appendix B. Feedback Construction and Finite Simulation Without Stabilization}
\phantomsection
\addcontentsline{toc}{section}{Appendix B. Feedback Construction and Finite Simulation Without Stabilization}
\label{app:sim-rl-coh-not-stable}

This appendix establishes an infinite-horizon feedback construction in which forecasts are conditionally correct under the governing law but empirical means fail to stabilize almost surely. It then interprets the supplied finite simulation as an illustration of the proved result.

\noindent \textbf{Environment, filtration, and conditional law.} Consider a one-state, two-action bandit. At time $t$, the learner chooses a predictable action $A_t\in\{0,1\}$ and then observes $X_t\in\{0,1\}$. Let
\[
\mathcal F_t:=\sigma(A_1,X_1,\ldots,A_t,X_t),
\]
so $A_t$ is $\mathcal F_{t-1}$-measurable. Fix $0<r_0<\tau<r_1<1$ and assume
\[
\mathbb P(X_t=1\mid\mathcal F_{t-1})=r_{A_t}
\qquad\text{a.s. for every }t.
\]
Together with the predictable update rule below, these Bernoulli transition kernels define a countably additive path law by the standard sequential product-kernel construction.

\noindent \textbf{Conditional correctness and martingale innovations.} Define
\[
p_{t-1}:=\mathbb E[X_t\mid\mathcal F_{t-1}]=r_{A_t}.
\]
Then $d_t:=X_t-p_{t-1}$ satisfies
\[
\mathbb E[d_t\mid\mathcal F_{t-1}] = 0\qquad\text{a.s. for all }t,
\]
so the bounded martingale strong law gives
\[
\frac{1}{n}\sum_{t=1}^n (X_t-p_{t-1})\to 0 \quad\text{a.s.}
\]
This is conditional-mean correctness and validity in the large under the governing law; it is not inferred merely from internal coherence under a separate assessment law.

\noindent \textbf{Infinite-horizon feedback rule.} For every $k\ge1$, let $L_k:=2^{k+L_0}$, $n_0:=0$, and $n_k:=\sum_{j=1}^kL_j=2^{L_0}(2^{k+1}-2)$. During episode $k$, set $A_t=a_k$ for $n_{k-1}<t\le n_k$. After observing
\[
\widehat R_k:=\frac1{L_k}\sum_{t=n_{k-1}+1}^{n_k}X_t,
\]
update
\[
a_{k+1}=
\begin{cases}
0, & \widehat{R}_k>\tau,\\
1, & \widehat{R}_k\le \tau,
\end{cases}.
\]
The rule switches away from action $1$ after a high episode average and toward action $1$ after a low one.

\noindent \textbf{Almost-sure switching and nonstabilization.} Let $\delta:=\min\{\tau-r_0,r_1-\tau\}>0$. Within episode $k$, the centered rewards are bounded martingale differences. A conditional Hoeffding--Azuma bound gives
\[
\mathbb P\!\left(a_{k+1}\ne1-a_k\mid\mathcal F_{n_{k-1}}\right)
\le2\exp\!\left(-\frac{\delta^2L_k}{2}\right).
\]
The right-hand side is summable in $k$. By the first Borel--Cantelli lemma, only finitely many threshold errors occur almost surely; hence the actions alternate from some finite episode onward.

At the episode endpoints,
\[
\frac1{n_k}\sum_{t=1}^{n_k}p_{t-1}
=\frac{\sum_{j=1}^kL_jr_{a_j}}{\sum_{j=1}^kL_j}.
\]
Because episode lengths double and actions eventually alternate, the asymptotic weight on episodes having the same action as episode $k$ is $2/3$, and the weight on the other action is $1/3$. Therefore, up to a possible parity swap caused by finitely many early threshold errors,
\[
\frac1{n_k}\sum_{t=1}^{n_k}p_{t-1}
\longrightarrow
\begin{cases}
\frac23r_1+\frac13r_0,&\text{along }\{k:a_k=1\},\\[2mm]
\frac23r_0+\frac13r_1,&\text{along }\{k:a_k=0\}.
\end{cases}
\]
The two limits differ. Since
\[
\bar X_n=\frac1n\sum_{t=1}^np_{t-1}+\frac1n\sum_{t=1}^n(X_t-p_{t-1})
\]
and the innovation term converges to zero almost surely, the same distinct subsequential limits hold for $\bar X_{n_k}$. Thus
\[
\mathbb P\!\left(\lim_{n\to\infty}\bar X_n\text{ exists}\right)=0.
\]
The finite plot in \autoref{fig:rl-coh-not-stable} illustrates this proved infinite-horizon separation; it does not itself establish the probability-one conclusion.

\begin{figure}[H]
\centering
\includegraphics[width=0.92\linewidth]{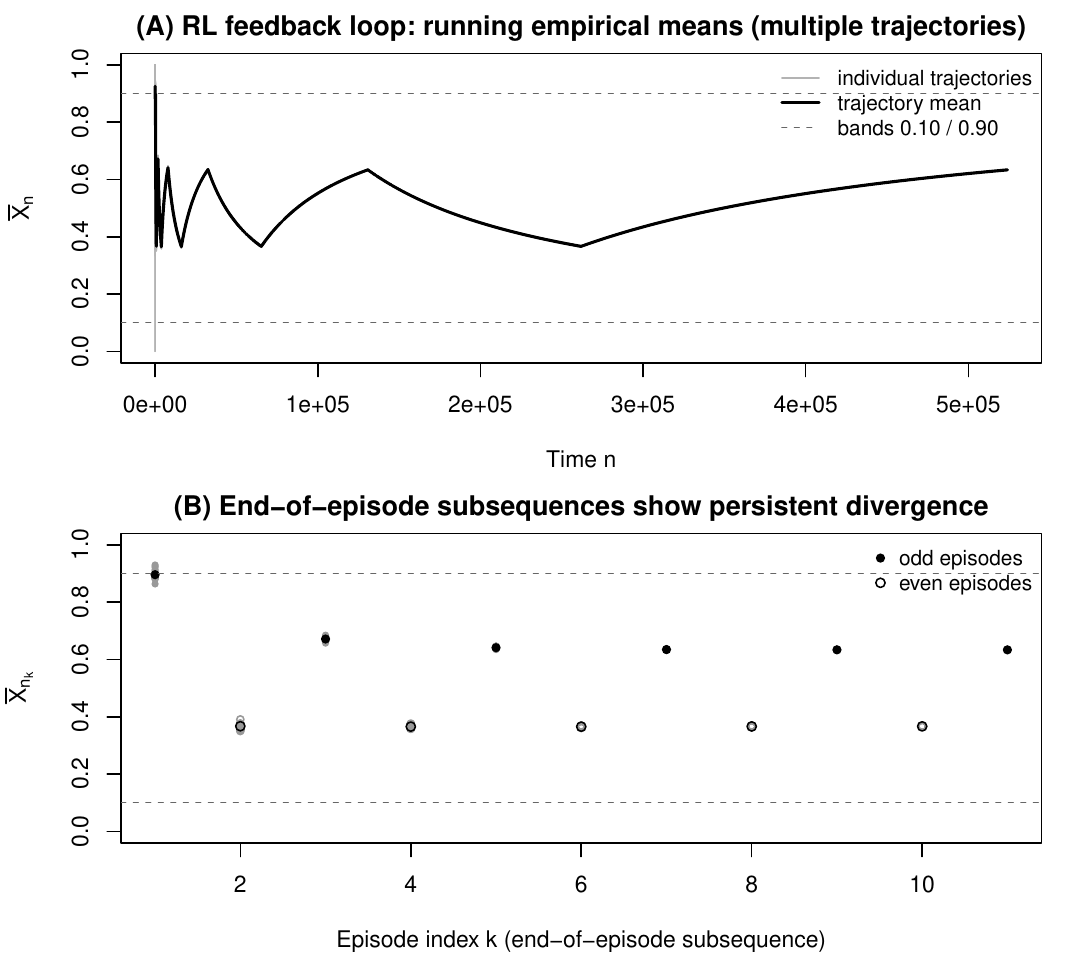}
\caption{\textbf{Finite simulation illustrating nonstabilization in the feedback construction.} Twenty trajectories were simulated with $r_0=0.10$, $r_1=0.90$, $\tau=0.50$, $a_1=1$, and $L_k=2^{k+7}$ for $K=11$ episodes, giving $n_{11}=524{,}032$. \textbf{(A)} Running empirical means $\bar X_n$ for individual trajectories (grey), their across-trajectory mean (black), and the action-conditional reward probabilities $r_0$ and $r_1$ (dashed). \textbf{(B)} Odd- and even-episode endpoint means remain separated over the displayed horizon. For the infinite-episode model proved above, the corresponding subsequences converge almost surely, up to a possible parity swap, to $\frac23r_1+\frac13r_0\approx0.633$ and $\frac23r_0+\frac13r_1\approx0.367$. The finite simulation illustrates rather than proves nonconvergence. Conditional-mean correctness holds because $\mathbb E[X_t\mid\mathcal F_{t-1}]=r_{A_t}$.}
\label{fig:rl-coh-not-stable}
\end{figure}

\refstepcounter{section}
\section*{Appendix C. Supplementary Examples}
\phantomsection
\addcontentsline{toc}{section}{Appendix C. Supplementary Examples}

This appendix collects additional examples that elaborate, extend, or operationalize the main arguments of the paper without being structurally necessary for the central results. These vignettes illustrate further instances of filtration mismatch, regime specification, and drift control across adaptive and AI-mediated settings. They are provided to support readers seeking additional technical detail or applied context, while keeping the main text focused on the Principle, the unified stabilization framework, and the core alignment and separation results.

\begin{exmp}[Optional stopping and why coherence is filtration-relative]
\label{ex:opt-stopping}    

Suppose an analyst monitors a running statistic $S_t$ and stops at a stopping time $\tau$. Let $\mathcal F_t$ include the realized history and stopping rule. Under an assessment law $\mathbb Q$, the conditional price $\mathbb E_{\mathbb Q}[Y\mid\mathcal F_t]$ remains an internally coherent assessment given that information. At each fixed $\mathcal F_t$, \autoref{prop:coherence-decision} excludes a nonzero one-sided payoff from finite $\mathcal F_t$-measurable stakes.

This fixed-information-set statement does not itself prove validity for accumulated multi-period trading strategies or invariance of an inferential error rate. Those claims require a specified martingale or test process, admissible predictable strategies, optional-sampling conditions, and a repetition regime that includes the stopping rule.
\end{exmp}

\medskip
\begin{exmp}[Adaptive experimentation and conditional optimality]
\label{ex:adp-expr}    

In an adaptive clinical trial, treatment assignments may depend on accumulated outcomes \citep{RobertsonEtAl2023}.
Let $\mathcal{F}_t$ denote the $\sigma$-algebra generated by assignments and responses up to time
$t$.  At time $t$, the trial may compute the posterior probability that treatment $A$ is superior
to treatment $B$:
\[
q_t := \mathbb{Q}(\theta_A > \theta_B \mid \mathcal{F}_t),
\]
and use this quantity to guide allocation or stopping decisions. Under its assessment law, $q_t$ is Bayes-optimal for the event $\{\theta_A>\theta_B\}$ under strictly proper scoring rules conditional on $\mathcal F_t$. Its interpretation remains conditional on that law and history.

However, whether nominal coverage or type-I error rates are preserved under the adaptive design is a distinct question.  Such frequency claims require a reference structure that specifies how the entire adaptive design is to be replicated.  Coherence ensures that the posterior is internally consistent; it does not by itself assert that repeated implementations of the design would satisfy
a particular long-run error bound.
\end{exmp}

\medskip
\begin{exmp}[Non-identifiability from a common Ces\`aro limit, without exchangeability]
\label{ex:freq-not-law}    
Let $\mathcal X=\{0,1\}^{\mathbb N}$ and $\bar X_n=n^{-1}\sum_{t=1}^n X_t$. We construct two distinct probability laws
$\mathbb P_1\neq \mathbb P_2$ on $\mathcal X$ such that $\bar X_n\to 1/2$ almost surely under both laws,
yet their finite-dimensional distributions differ. In particular,
\[
\mathbb{P}_1(X_1=X_2)=\frac{1}{2}
\qquad\text{but}\qquad
\mathbb{P}_2(X_1=X_2)=\frac{1}{10}.
\]

\textbf{Law $\mathbb P_1$ (i.i.d.\ Bernoulli).}
Under $\mathbb P_1$, the coordinate process $(X_t)_{t\ge1}$ is i.i.d.\ with $\mathbb P_1(X_t = 1) = 1/2$.
Thus for every $n \ge 1$ and $(x_1,\dots,x_n)\in\{0,1\}^n$,
\[
\mathbb P_1(X_1=x_1,\dots,X_n=x_n)=2^{-n}.
\]
By the strong law of large numbers, $\bar X_n\to 1/2$ almost surely. Moreover,
$\mathbb{P}_1(X_1=X_2)=\mathbb{P}_1(X_1=0,X_2=0)+\mathbb{P}_1(X_1=1,X_2=1)=1/4+1/4=1/2$.

\textbf{Law $\mathbb P_2$ (ergodic Markov dependence).}
Under $\mathbb P_2$, $(X_t)_{t\ge1}$ is a stationary, irreducible, aperiodic two-state Markov chain with transition matrix
\[
P
=
\begin{pmatrix}
\mathbb P_2(X_{t+1}=0 \mid X_t=0) &
\mathbb P_2(X_{t+1}=1 \mid X_t=0)
\\
\mathbb P_2(X_{t+1}=0 \mid X_t=1) &
\mathbb P_2(X_{t+1}=1 \mid X_t=1)
\end{pmatrix}
=
\begin{pmatrix}
1/10 & 9/10 \\
9/10 & 1/10
\end{pmatrix}.
\]
Initialize the chain in its stationary distribution $\pi=(1/2,1/2)$. Then $\mathbb P_2(X_t=1)=1/2$ for all $t$, and
\[
\mathbb P_2(X_1=x_1,\dots,X_n=x_n)
=
\pi_{x_1}\prod_{t=1}^{n-1}P_{x_t,x_{t+1}}.
\]
By the ergodic theorem for Markov chains, $\bar X_n\to \mathbb E_{\mathbb P_2}[X_1]=1/2$ almost surely.
However,
\[
\mathbb{P}_2(X_1=X_2)=\sum_{x\in\{0,1\}}\pi_x P_{x,x}=\frac{1}{2}\cdot\frac{1}{10}+\frac{1}{2}\cdot\frac{1}{10}=\frac{1}{10},
\]
so $\mathbb P_2\neq \mathbb P_1$ despite the common almost sure Ces\`aro limit.
\end{exmp}

\medskip
\begin{exmp}[Recommendation algorithms and policy–environment mixtures]
\label{ex:rec-alg}    

Suppose a recommendation system deploys a fixed policy $\pi$ that maps user histories to item
exposures.  Let $X_t\in\{0,1\}$ denote whether the displayed item is clicked.
Even if empirical click-through rates $\bar X_n$ converge to $p$, distinct user-type mixtures, preference distributions, or exposure mechanisms can yield the same aggregate rate under policy $\pi$. The induced laws can nevertheless differ in finite-horizon predictive distributions. The limiting frequency identifies at most one policy-indexed performance functional of the deployed pair, not the pair itself. Behavior under alternative policies additionally requires a specified causal or structural intervention model and its identification assumptions.
\end{exmp}

\medskip
\begin{exmp}[Adaptive clinical trials and regime dependence]
\label{ex:adp-clin}    
In an adaptive clinical trial, treatment allocation probabilities evolve with accumulating data.
Suppose that under a given adaptive rule, the empirical proportion of patients receiving the
superior treatment converges to a constant $p$.  This convergence reflects the joint behavior of
the learning rule and the data-generating mechanism. However, multiple distinct stochastic models for patient response, differing in latent heterogeneity or time-varying effects, may induce the same limiting allocation proportion under the
same adaptive design.  The long-run frequency of assignments or outcomes therefore does not
identify the full probability law governing patient responses.  It captures only a marginal
asymptotic consequence of the interaction between design and environment.
\end{exmp}

\medskip
\begin{exmp}[A ``freeze-out'' regime: stabilized conditionally correct forecasts yield stabilized averages]
\label{ex:freezeout}
Consider an adaptive system that continues to observe a data stream $(Z_t)$ but, after a burn-in
period, stops changing the internal state used to form one-step forecasts.  This captures a common
deployment pattern in which an online learner explores and updates early, then freezes a model or
policy version for a stable production window.

Formally, let $(\mathcal{F}_t)_{t\ge0}$ be the filtration generated by the realized history, and let
$(X_t)_{t\ge1}$ be a bounded outcome of interest (e.g., click/no-click), with $X_t\in[0,1]$.
Suppose the forecaster's state $S_t$ takes values in a metric space $\mathsf S$, is $\mathcal F_t$-measurable, and converges almost surely to $S_\infty$. Assume either that $S_t$ is eventually constant almost surely or that $g:\mathsf S\to[0,1]$ is measurable and continuous at $S_\infty$ almost surely. Let the evaluation-law conditionally correct forecast be
\[
p_t:=\mathbb E_{\mathbb P}[X_{t+1}\mid\mathcal F_t]=g(S_t).
\]
Then $p_t\to p_\infty:=g(S_\infty)$ almost surely. By
Proposition~\ref{prop:filtration-stabilization}, the empirical mean stabilizes to the same limit
\[
\bar X_n := \frac{1}{n}\sum_{t=1}^n X_t \xrightarrow{\mathrm{a.s.}} p_\infty.
\]
Thus, in a freeze-out deployment window, \emph{stabilization of governing-law conditionally correct forecasts} is a
structural route to long-run stabilization of empirical averages (frequency meaning \hyperref[f:F1]{(\ref*{f:F1})}).

This conclusion should not be over-interpreted.  The limit $p_\infty$ is a regime-conditional
quantity tied to the realized filtration and the deployed policy/model version.  Moreover, \hyperref[f:F1]{(\ref*{f:F1})}
stabilization alone does not guarantee \hyperref[f:F3]{(\ref*{f:F3})} calibration; calibration depends on the checking classes and repetition protocol and can fail under selection effects or distribution shift even if $\bar X_n$ stabilizes
(see \autoref{rem:vitl-vs-calibration}).  The example therefore illustrates the separation
principle in a constructive direction: evaluation-law conditional correctness plus a regime that yields
$p_t\to p_\infty$ produces stabilization, while stronger reliability claims require an explicit
evaluation protocol and reference structure.
\end{exmp}

\medskip
\begin{exmp}[Stationary but not ergodic: a latent-mixture regime yields a random limit]
\label{ex:nonergodic}
Let $Z\sim\mathrm{Bernoulli}(1/2)$ and fix two distinct success probabilities $p_0,p_1\in(0,1)$
(e.g., $p_0=0.2$ and $p_1=0.8$). Conditional on $Z=z$, let $(X_t)_{t\ge1}$ be i.i.d.\
$\mathrm{Bernoulli}(p_z)$:
\[
\mathbb{P}(X_t=1\mid Z=z)=p_z,\qquad t\ge1,
\]
and assume conditional independence across $t$ given $Z$.
Unconditionally, $(X_t)$ is strictly stationary (a mixture of i.i.d.\ product measures), but it is
not ergodic because the latent regime variable $Z$ is shift-invariant and generates a nontrivial invariant
$\sigma$-algebra $\mathcal{I}=\sigma(Z)$.

Let $\bar X_n:=n^{-1}\sum_{t=1}^n X_t$. By the strong law of large numbers applied conditional on
$Z$,
\[
\bar X_n \xrightarrow{\mathrm{a.s.}} p_Z,
\]
so the stabilized limit is \emph{regime-conditional} and takes the random values $p_0$ or $p_1$.
Moreover, since the sequence $(X_t)$ is conditionally i.i.d.\ given $Z$, we have
\[
\mathbb{E}[X_1\mid \mathcal{I}]
=\mathbb{E}[X_1\mid Z]
=p_Z,
\]
which matches the almost sure limit in Birkhoff's theorem for stationary (not necessarily ergodic)
processes. 
\end{exmp}

\medskip
\begin{exmp}[Bandit special case: stabilization is policy-indexed even in a fixed environment]
\label{ex:bandit}
Consider a two-armed bandit with rewards $(X_t)_{t\ge1}$ and actions $A_t\in\{0,1\}$. Assume the conditional reward mean is stable given the past and current action:
\[
\mathbb E[X_t\mid\mathcal F_{t-1},A_t]=\mu_{A_t}\qquad\text{a.s.},
\]
and $0\le X_t\le 1$ almost surely.  Fix a stationary randomized policy that pulls arm $1$ with probability
$q\in(0,1)$ at each round, independently of the past,
\[
\mathbb{P}(A_t=1\mid \mathcal{F}_{t-1})=q,\qquad t\ge1,
\]
where $\mathcal{F}_{t-1}$ is the realized history up to time $t-1$.

Let $\bar X_n:=n^{-1}\sum_{t=1}^n X_t$.  Under this regime, the reference drift in the unified
framework (Section~\ref{sec6.1-unified}) is
\[
m_t=\mathbb{E}[X_t\mid \mathcal{F}_{t-1}]
=\mathbb{E}\bigl[\mu_{A_t}\mid \mathcal{F}_{t-1}\bigr]
=q\,\mu_1+(1-q)\,\mu_0,
\]
a constant.  Hence Proposition~\ref{prop:unified-stabilization} yields
\[
\bar X_n \xrightarrow{\mathrm{a.s.}} q\,\mu_1+(1-q)\,\mu_0.
\]
The stabilized frequency is therefore policy-indexed, as changing the policy parameter $q$
changes the limiting average reward even though the environment (the arm means $\mu_0,\mu_1$) is
fixed.  This is the simplest instance of the regime dependence made explicit in
Proposition~\ref{prop:policy-stabilization}.
\end{exmp}

\medskip
\begin{exmp}[Adaptive A/B testing: drift stabilization is a design-level condition]
\label{ex:ab}
In an adaptive A/B test, the assignment probability $q_t$ may be updated as outcomes accrue, for example, to increase allocation to the better-performing arm, while inference remains grounded in the known experimental design. This setting includes response-adaptive randomization schemes that progressively tilt allocation toward the empirically superior arm. Two distinct consequences follow:

\smallskip

\noindent\emph{(i) Stabilization is a design-level condition.}
If the adaptive rule and potential-outcome sequence are such that the Ces\`aro drift stabilizes,
$\frac{1}{n}\sum_{t=1}^n m_t\to\theta$ almost surely, then
Proposition~\ref{prop:design-stabilization} yields $\bar X_n\to\theta$ almost surely.  In this
sense, stabilization of the observed average is guaranteed not by ``i.i.d.\ sampling'' but by an
explicit property of the design-induced drift.

\smallskip
\noindent\emph{(ii) Anytime-valid monitoring is design-based.}
The martingale-difference identity alone does not produce an observable test process because its centering term may involve unobserved potential outcomes. If an observable nonnegative supermartingale or e-process is proved valid under a specified null and the known randomization law, Ville's inequality yields time-uniform control. Consequently, optional stopping does not invalidate that particular design-based guarantee, precisely because its null, process, and randomization mechanism are stated explicitly.
\end{exmp}

\medskip
\begin{exmp}[LLM context truncation as a filtration mismatch]
\label{ex:mode1}
In a long conversation, a user may interpret the model’s output at time $t$ as conditioned on the
full transcript $\mathcal{F}_t$. In practice, an LLM may base its response on a truncated context window $W_t$. Let $\mathcal I_t:=\sigma(W_t)$ denote the model's effective information set. Because a sliding window can discard information, $(\mathcal I_t)$ need not be a filtration. A reported score $p_t$ then need not be interpretable as
$\mathbb{E}[X_{t+1}\mid\mathcal{F}_t]$ for the declared filtration, and apparent time-inconsistent
updates can result without any new information.
\end{exmp}

\medskip
\begin{exmp}[Coherence is not the same as a usable frequency guarantee]
\label{ex:mode2}
The RL feedback construction in \autoref{ex:coh-not-stable} is deliberately extreme but clarifies the diagnosis.
There, the forecasts are conditionally correct under the governing law (and are internally coherent if that law is also used as the assessment law), and the innovations satisfy validity in the large \eqref{eq:vitl}, yet the empirical
mean fails to converge because the predictable component (reference drift) does not stabilize.  This is a Mode~2
failure in which the one-step probabilistic object is well defined, but the intended frequency
meaning \hyperref[f:F1]{(\ref*{f:F1})} is not justified without additional regime assumptions that control drift.
\end{exmp}

\medskip
\begin{exmp}[Product support assistant: fluent troubleshooting text is not a fault-diagnosis model]
\label{ex:halluc-cite}
A manufacturer deploys an LLM-based support agent to help customers troubleshoot a smart thermostat.  Given a chat history
and a symptom description, the agent outputs a recommended fix $O_{t+1}$ (e.g., ``replace the C-wire'' or ``reset the
firmware'') along with a scalar confidence score $p_t$ derived from generation probabilities. An autoregressive distribution over strings does not by itself identify or supply a conditional probability that the recommendation is correct. Such a probability requires a jointly specified model containing the external truth or evaluation variables and the mapping from output to correctness. Customers and operators nevertheless interpret $p_t$ as the probability that the recommendation is \emph{correct} for the
device's true failure state.  Formally, they demand a probability for the event
\[
X_{t+1}:=\mathbf 1\{\text{$O_{t+1}$ is judged the correct fix under the true latent fault and service protocol}\}.
\]
This event is defined relative to the latent fault, sensor logs, and service rules. A coherent joint law could be constructed on an extended space containing these variables, but the text model alone does not provide it. Thus a high string-based score can be unrelated to $\E_{\mathbb P}[X_{t+1}\mid\mathcal F_t]$ for operational correctness.
\end{exmp}

\medskip
\begin{exmp}[Regime drift in an LLM-assisted medical triage pipeline]
\label{ex:feedback-queries}
A hospital deploys an LLM-assisted triage agent that recommends a disposition and outputs a score $\widetilde p_t$ intended to approximate the evaluation-law conditional mean $p_t^*:=\E_{\mathbb P}[X_{t+1}\mid\mathcal F_t]$, where
\[
X_{t+1}:=\mathbf 1\{\text{the disposition is judged correct by a clinician review}\}
\]
and $\mathcal F_t$ includes the questionnaire, vitals, retrieved guidelines, and prior actions. For the exact $p_t^*$, the innovations $X_{t+1}-p_t^*$ form a $\mathbb P$-martingale difference. For the approximate score,
\[
\frac1n\sum_{t=0}^{n-1}(X_{t+1}-\widetilde p_t)
=\frac1n\sum_{t=0}^{n-1}(X_{t+1}-p_t^*)
+\frac1n\sum_{t=0}^{n-1}(p_t^*-\widetilde p_t),
\]
so average approximation error must also be controlled. Even if $\widetilde p_t=p_t^*$ exactly, staff overrides, protocol revisions, seasonality, and model updates can change visitation of confidence bins and prevent raw within-bin frequencies from stabilizing. Internal coherence under a separate assessment law would not by itself imply the $\mathbb P$-martingale identity.
\end{exmp}

\medskip
\begin{exmp}[Ad click prediction: global unbiasedness without binwise calibration (C1)]
\label{ex:global-not-cal}
An online advertising platform reports a probability-like score $p_t\in\{0.1,0.9\}$ for the click indicator $X_{t+1}\in\{0,1\}$. The score is not assumed here to be the exact evaluation-law conditional mean $\E_{\mathbb P}[X_{t+1}\mid\mathcal F_t]$; the example illustrates how a global finite-window discrepancy can cancel while binwise discrepancies remain large. Across a week of traffic, suppose the platform serves an equal number of impressions with $p_t=0.1$ and $p_t=0.9$, and the realized click frequencies are $\hat{\mathbb{P}}(X_{t+1}=1\mid p_t=0.1)=0.2$ and $\hat{\mathbb{P}}(X_{t+1}=1\mid p_t=0.9)=0.8$ (e.g., the model is systematically underconfident on ``low'' cases and overconfident on ``high'' cases).
Then the global discrepancy cancels:
\[
\frac1n\sum_{t=1}^n (X_{t+1}-p_t)\approx \tfrac12(0.2-0.1)+\tfrac12(0.8-0.9)=0,
\]
so the aggregate discrepancy can be close to zero over this audit window, while fixed-bin reliability is poor on both bins. The low bin has frequency $0.2$ vs.\ mean forecast $0.1$, and the high bin has frequency $0.8$ vs.\ $0.9$.
This illustrates that a single global cancellation statement does not imply calibration, which requires binwise
(or otherwise conditional) discrepancies to vanish.
\end{exmp}

\medskip
\begin{exmp}[Fraud-risk scores: calibration depends on the checking class and the deployment regime]
\label{ex:selective}
A payment processor assigns each transaction a fraud-risk score $p_t$ and flags transactions when $p_t$ exceeds a moving
threshold; flagged transactions trigger extra verification, which reduces realized fraud among flagged cases.
The compliance team reports that ``the model is calibrated'' based on last month's retrospective sample.

In fact, the claim is incomplete in two ways.  First, it depends on the checking class $\mathcal C$. The score may be
well-calibrated in coarse probability bins overall, yet miscalibrated conditional on merchant category, country, or
verification pathway (subgroup checks are different elements of $\mathcal C$).  Second, it depends on the reference
structure $\mathfrak R$. Once the threshold policy, verification workflow, or model weights are updated, the joint process
$(p_t,X_{t+1})$ changes, so calibration statements from last month do not automatically transfer to the new regime.
A meaningful statement must specify both (i) the checking protocol (bins/subgroups/selection rules) and (ii) the regime
under which repetition is defined (e.g., a frozen model and fixed intervention policy over a declared audit window).
\end{exmp}

\medskip
\begin{exmp}[Matched marginals can hide wrong conditional structure]
\label{ex:synth-cond}
A state agency releases synthetic microdata for an unemployment-benefits program.
The real data law $\mathbb P$ includes covariates $X$ (age, industry, county), a treatment indicator $T$ (participation in
a job-training program), and an outcome $Y$ (reemployment within 12 weeks).
The agency verifies that $\mathbb Q$ matches the univariate marginals and all pairwise correlations under $\mathbb P$. In particular, for nondegenerate binary $T$ and $Y$, exact matching of their marginals and correlation determines their $2\times2$ joint distribution and therefore preserves the unadjusted mean contrast $\E[Y\mid T=1]-\E[Y\mid T=0]$.

The analyst instead studies the conditional associational contrast
\[
\Delta(x):=\E_{\mathbb P}[Y\mid T=1,X=x]-\E_{\mathbb P}[Y\mid T=0,X=x]
\]
and an adjusted summary $\Delta_w:=\int\Delta(x)\,w(dx)$ for a prespecified target covariate distribution $w$ supported on the common conditional support. The pairwise checks do not determine the higher-order conditional law $\mathbb P(Y\mid T,X)$. A generator may preserve every checked pairwise summary while smoothing a three-way or higher-order $T$--$X$--$Y$ interaction. The corresponding quantities under $\mathbb Q$ can therefore differ from those under $\mathbb P$, although the unadjusted $T$--$Y$ contrast is preserved. Causal interpretations would additionally require potential-outcome and identification assumptions not imposed here.
\end{exmp}

\medskip
\begin{exmp}[Cross-regime validity fails under retraining and feedback loops]
\label{ex:synth-drift}
A hospital system uses a proprietary synthetic-data service to share patient records with external researchers.
Each quarter, the vendor retrains the generator on the latest electronic health record data, so the synthetic law
$\mathbb Q^{(q)}$ changes over quarters $q$.
Researchers use each quarterly release to build and validate a risk model and report ``95\% coverage'' for confidence
intervals of a population-level estimand (e.g., 30-day readmission rate) computed on the synthetic data.

The coverage statement is a frequency claim whose meaning depends on the replication scheme. Under one conditional scheme, coverage is evaluated over repeated synthetic draws while holding fixed the training data, privacy mechanism, and fitted generator. Other schemes may also repeat the training sample or retraining procedure and therefore define different coverage claims. In the deployed example, none of these components remains fixed across quarters. Retraining changes $\mathbb Q^{(q)}$, and the hospital uses findings derived from the synthetic releases to modify care pathways, which in turn alters the real data law $\mathbb P$ collected next quarter. Thus the system operates in a coupled, adaptive regime where both $\mathbb P$ and $\mathbb Q^{(q)}$ drift. Without an explicit reference structure linking repetitions under $\mathbb Q$ to repetitions under $\mathbb P$ (e.g.,
a frozen generator over a declared audit window, or a design-based retraining protocol treated as part of $\mathfrak R$),
``coverage under $\mathbb Q$'' does not justify ``coverage for the real-world estimand under $\mathbb P$.'' 
\end{exmp}

\refstepcounter{section}
\section*{Appendix D. Representative Reference Structures and Checklist for Reporting Probabilistic Guarantees in Adaptive Regimes}
\phantomsection
\addcontentsline{toc}{section}{Appendix D. Representative Reference Structures and Checklist for Reporting Probabilistic Guarantees in Adaptive Regimes}
\label{appendix-D}

\subsection*{D.1 Reference Structures and Sufficient Stability Conditions}
\phantomsection
\label{appendix-D1}

\begin{table}[H]
\centering
\captionsetup{font=footnotesize}
\caption{Reference structures and sufficient stability conditions. A reference structure alone does not license the listed conclusion; the target- and procedure-specific assumptions in the final column are also required.}
\label{tab:ref-structures}
\footnotesize
\renewcommand{\arraystretch}{1.25}  

\begin{tabular}{@{}L{0.21\linewidth} L{0.37\linewidth} L{0.36\linewidth}@{}}
\toprule
Reference structure & What is held fixed / invariance & Typical frequency meaning supported \\
\midrule
Exchangeability & Invariance under finite permutations of observations & \hyperref[f:F1]{(\ref*{f:F1})} with latent limit; alignment of conditionally correct forecasts and frequencies \\
Stationary/ergodic regime & Shift invariance of the law; ergodicity strengthens to constant limits & \hyperref[f:F1]{(\ref*{f:F1})} time-average stabilization; regime-conditional limits \\
Fixed policy--environment & Policy $\pi$ and environment kernel $\mathsf P$ fixed; positive-Harris ergodicity & Long-run average-reward stabilization for that pair \\
Randomized design & Assignment mechanism and held-fixed potential outcomes specified & \hyperref[f:F2]{(\ref*{f:F2})} coverage or error control for a specified procedure proved valid under the design and its assumptions \\
Controlled drift / reference drift & Drift term (reference conditional mean) has convergent Ces\`aro averages & \hyperref[f:F1]{(\ref*{f:F1})} stabilization via the unified drift--innovation template \\
Ideal conditional-mean forecasts with fixed predictable bins & $p_{t-1}=\mathbb E_{\mathbb P}[X_t\mid\mathcal F_{t-1}]$; predeclared Borel bin and infinite visitation & \hyperref[f:F3]{(\ref*{f:F3})} fixed-bin calibration under the stated visit condition \\
\bottomrule
\end{tabular}
\normalsize
\end{table}

\subsection*{D.2 Checklist for Reporting Probabilistic Guarantees in Adaptive Regimes}
\phantomsection
\label{app:reporting-checklist}

For any probabilistic claim, especially one intended to support a frequency meaning
\hyperref[f:F1]{(\ref*{f:F1})}--\hyperref[f:F3]{(\ref*{f:F3})}, the following items should be stated
explicitly to make the claim auditable in adaptive regimes:

\begin{enumerate}
\item \textbf{Name the target event/statistic.} Specify $X_{t+1}$ (or $X_t$) and the evaluation rule
(e.g., factual correctness relative to a database; reward under a policy; coverage for an estimand).
\item \textbf{Declare the filtration.} State what information is included in $\mathcal{F}_t$
(prompts, retrieved documents, tool outputs, human feedback, selection/abstention indicators, model version).
\item \textbf{State the probabilistic object and law.} Clarify whether $p_t$ is an assessment-law conditional probability, an evaluation-law conditionally correct forecast, or a heuristic score; if the latter, label it as a score.
\item \textbf{Declare the frequency meaning.} Identify whether the claim concerns stabilization
\hyperref[f:F1]{(\ref*{f:F1})}, repeated-sampling or time-uniform guarantees
\hyperref[f:F2]{(\ref*{f:F2})}, or calibration
\hyperref[f:F3]{(\ref*{f:F3})}.  For calibration, state the checking class (bins/groups/rules).
\item \textbf{Declare the reference structure.} Specify what is held fixed under repetition (policy, environment,
design, stopping rule, model version, generator protocol) and what invariance/controlled evolution is assumed.
\item \textbf{Give the stability route.} Point to the theorem-level mechanism supporting the claim. Examples include exchangeability, predictive stabilization, shift-based stabilization, policy--environment ergodicity, and a specified design-based or time-uniform construction (\autoref{sec6-beyond-exch}).
\item \textbf{State failure modes.} Describe which violations would break the claim (drift, selection effects,
model updates, strategic environments, mismatch between modeled events and evaluated events).
\end{enumerate}

\medskip

These criteria do not eliminate judgment. In many deployments, the operative reference structure is partially observed
or contested, and some stability assumptions can only be defended approximately.  The point is that, in adaptive
regimes, probabilistic reliability is not a property of ``the model'' alone.  It is a property of a coupled system
and therefore must be reported as such.

\refstepcounter{section}
\section*{Appendix E. Applied Illustrations Under a Declared Reference Structure}
\phantomsection
\addcontentsline{toc}{section}{Appendix E. Applied Illustrations Under a Declared Reference Structure}
\label{appendix-E}

This appendix develops two complementary illustrations of the Role Separation Principle in a
closed-loop recommendation system.  The first illustration focuses on global stabilization and policy-value
estimation, showing how frequency claims of types \hyperref[f:F1]{(\ref*{f:F1})}–\hyperref[f:F2]{(\ref*{f:F2})}
become meaningful only under an explicitly declared repetition regime.
The second illustration isolates fixed-bin calibration as an \hyperref[f:F3]{(\ref*{f:F3})} statement, emphasizing that its validity depends on regime-level visitation conditions.

The audit cohort is not proposed as a novel evaluation mechanism. The purpose of the illustrations is structural and semantic: familiar logging and frozen-model devices are interpreted as components of an explicit reference regime.

In much of the contextual bandit literature, logging policies and frozen models are introduced primarily as devices for unbiased estimation. Our illustrations make explicit that stabilization and calibration are not properties of the forecasting model alone; they are properties of the declared regime linking the filtration of interpretation to the structure of repetition.

\subsection*{E.1 Global Stabilization and Policy Value Under an Explicit Audit Regime}
\phantomsection
\label{sec:applied-E1}

\noindent \textbf{System and filtration.}
We consider a recommendation platform modeled as a contextual bandit,
i.e., a sequential decision problem in which, at each time $t$,
an action is chosen based on observed covariates and a single stochastic
outcome is observed.

At time $t$, let
\begin{itemize}
\item $U_t$ denote the observed context (covariates describing the user and situation);
\item $C_t$ denote the finite candidate set of available items;
\item $A_t \in C_t$ denote the chosen action (displayed item);
\item $X_t \in [0,1]$ denote the observed reward (e.g., click indicator).
\end{itemize}
The platform deploys an adaptive production policy whose internal state $\theta_t$ is updated using past interactions. Thus $\theta_t$ is part of the evolving information.

Let the post-outcome history be
\[
\mathcal F_t
=
\sigma(U_1,C_1,A_1,X_1,\theta_1,\ldots,U_t,C_t,A_t,X_t,\theta_t).
\]
Define the preassignment and preoutcome information sets
\[
\mathcal H_t:=\mathcal F_{t-1}\vee\sigma(U_t,C_t),
\qquad
\mathcal G_t:=\mathcal H_t\vee\sigma(A_t).
\]
The chronological nesting is $\mathcal F_{t-1}\subseteq\mathcal H_t\subseteq\mathcal G_t\subseteq\mathcal F_t$. Under a governing law, define the conditionally correct preoutcome forecast
\[
p_{t-1}
:=
\mathbb E[X_t\mid \mathcal G_t].
\]
Then $\mathbb E[X_t-p_{t-1}\mid\mathcal G_t]=0$. Because $X_t-p_{t-1}$ is $\mathcal F_t$-measurable and $\mathcal F_{t-1}\subseteq\mathcal G_t$, the tower property also gives $\mathbb E[X_t-p_{t-1}\mid\mathcal F_{t-1}]=0$; hence the partial sums form a martingale with respect to the nested post-outcome filtration $(\mathcal F_t)$.

\noindent \textbf{Innovation–drift decomposition.}
The click-through rate satisfies
\begin{equation}
\label{eq:E1-decomp}
\bar X_n
=
\frac1n\sum_{t=1}^n p_{t-1}
+
\frac1n\sum_{t=1}^n (X_t-p_{t-1}).
\end{equation}
Martingale laws control the innovation term.
Stabilization of $\bar X_n$ therefore depends on the preoutcome conditional-mean component,
which is regime-level.
In the production loop, $(U_t,C_t,A_t)$ is policy-dependent and
$\theta_t$ evolves, so the preoutcome conditional-mean component may drift.
Conditional-mean correctness alone does not ensure stabilization of global performance metrics.

\noindent \textbf{Declared reference structure: audit cohort.}
To make policy-value and click-through-rate statements frequency-meaningful, the platform declares an infinite audit sequence governed by a law $\mathbb P_{\mathfrak R}$, with expectation $\mathbb E_{\mathfrak R}$; $n$ denotes an increasing reporting horizon:

\begin{quote}
\textbf{Audit reference structure $\mathfrak R$.}
The logging policy $\nu(\cdot\mid U,C)$ and audit model $\theta^{\mathrm{audit}}$ are fixed. For each action $a\in C_t$, let $X_t(a)\in[0,1]$ be the potential reward. Assume: (i) consistency and no cross-interaction interference, $X_t=X_t(A_t)$; (ii) logging randomization,
\[
\mathbb P_{\mathfrak R}\!\left(A_t=a\mid\mathcal H_t,\{X_t(b):b\in C_t\}\right)
=\nu(a\mid U_t,C_t);
\]
(iii) conditional-mean stability for one measurable function $\mu$,
\[
\mathbb E_{\mathfrak R}[X_t(a)\mid\mathcal H_t]=\mu(U_t,C_t,a);
\]
(iv) positivity with uniformly bounded importance weights,
\[
\sup_{t,a:\,\pi(a\mid U_t,C_t)>0}
\frac{\pi(a\mid U_t,C_t)}{\nu(a\mid U_t,C_t)}\le W<\infty;
\]
and (v) $(U_t,C_t)_{t\ge1}$ is stationary and ergodic under $\mathbb P_{\mathfrak R}$.
\end{quote}

These are modeling assumptions.
They need not hold automatically in deployment and may fail under
inventory shifts, seasonality, or cross-cohort interference.
The Principle does not assert that $\mathfrak R$ is true;
it requires that such a regime be declared when frequency claims are made.

\noindent \textbf{Policy value under $\mathfrak R$.}
Let $\pi$ be a target policy with support contained in that of $\nu$
(i.e., $\pi\ll\nu$).
Define
\[
v_\pi(u,c):=\sum_{a\in c}\pi(a\mid u,c)\mu(u,c,a),
\qquad
V_{\mathfrak R}(\pi):=\mathbb E_{\mathfrak R}[v_\pi(U,C)].
\]
Let
\[
Y_t:=\frac{\pi(A_t\mid U_t,C_t)}{\nu(A_t\mid U_t,C_t)}X_t,
\qquad
\widehat V_n(\pi):=\frac1n\sum_{t=1}^nY_t.
\]
The assumptions give
\[
\mathbb E_{\mathfrak R}[Y_t\mid\mathcal H_t]=v_\pi(U_t,C_t).
\]
Consequently,
\[
\widehat V_n(\pi)-V_{\mathfrak R}(\pi)
=\frac1n\sum_{t=1}^n\{Y_t-v_\pi(U_t,C_t)\}
+\left\{\frac1n\sum_{t=1}^nv_\pi(U_t,C_t)-V_{\mathfrak R}(\pi)\right\}.
\]
The first term converges almost surely to zero by the bounded martingale strong law; the second converges almost surely to zero by the ergodic theorem. Hence
\[
\widehat V_n(\pi)\xrightarrow{\mathrm{a.s.}}V_{\mathfrak R}(\pi).
\]
Separate time-uniform inference would require an observable nonnegative supermartingale or e-process valid under a specified null and additional measurability and moment or conditional-MGF conditions.

This stabilization result holds for the audit regime.
Generalizing to the evolving production system requires additional
assumptions linking production and audit environments.
The audit cohort thus defines a distinct repetition regime
under which frequency claims are licensed.

\subsection*{E.2 Fixed-Bin Calibration as a Regime-Conditional Guarantee}
\phantomsection
\label{sec:applied-E2}

We continue with the recommender system introduced in the first illustration and turn specifically to the calibration question.

Under $\mathbb P_{\mathfrak R}$, let $p_{t-1}:=\mathbb E_{\mathfrak R}[X_t\mid\mathcal G_t]$. Fix a Borel bin $I\subset[0,1]$ and define
\[
W_t := \mathbf 1\{p_{t-1}\in I\},
\qquad
N_n(I)=\sum_{t=1}^n W_t.
\]
Let
\[
\widehat{x}_n(I)
=
\frac{1}{N_n(I)}\sum_{t=1}^n W_t X_t,
\qquad
\widehat{p}_n(I)
=
\frac{1}{N_n(I)}\sum_{t=1}^n W_t p_{t-1}.
\]
Then
\[
\widehat{x}_n(I)-\widehat{p}_n(I)
=
\frac{\sum_{t=1}^n W_t(X_t-p_{t-1})}{\sum_{t=1}^n W_t}.
\]

\noindent \textbf{Filtration-level result.}
Set $\mathcal K_t:=\mathcal G_{t+1}$. Then $(\mathcal K_t)$ is nested, $X_t$ is $\mathcal K_t$-measurable, and $p_{t-1}=\mathbb E_{\mathfrak R}[X_t\mid\mathcal K_{t-1}]$. By \autoref{prop:fixed-bin-calibration},
\[
\widehat{x}_n(I)-\widehat{p}_n(I)
\;\xrightarrow{\mathrm{a.s.}}\;0
\quad
\text{on the event } \{N_n(I)\to\infty\}.
\]
Conditional-mean correctness under $\mathbb P_{\mathfrak R}$ therefore yields a fixed-bin calibration guarantee
conditional on infinite visitation.

\noindent \textbf{Production obstruction.}
In the adaptive production loop, visitation is policy-dependent.
Forecast distributions may drift or concentrate as $\theta_t$ evolves,
so some bins may be visited only finitely often.
Conditional-mean correctness does not ensure $\{N_n(I)\to\infty\}$.
Calibration can thus fail to admit a stable frequency interpretation
even when forecasts are conditionally correct under the governing law (and internally coherent when the assessment and governing laws agree).

\noindent \textbf{Audit-regime restoration.}
Under the declared audit law $\mathbb P_{\mathfrak R}$, assume $(p_{t-1})_{t\ge1}$ is stationary and ergodic and
\[
q_I
:=
\mathbb P_{\mathfrak R}(p_{t-1}\in I)>0.
\]
Birkhoff's theorem gives
\[
\frac{N_n(I)}{n}
=\frac1n\sum_{t=1}^n\mathbf1\{p_{t-1}\in I\}
\xrightarrow{\mathrm{a.s.}}q_I>0,
\]
so $N_n(I)\to\infty$.
Combining this with \autoref{prop:fixed-bin-calibration}
yields a stable calibration conclusion within the audit cohort.

This statement is explicitly regime-conditional.
It does not assert that production calibration stabilizes;
it asserts that calibration is interpretable under the declared
repetition structure.

\par
}
\clearpage 
\end{singlespace}
\end{document}